\documentclass[letter, 10pt, conference]{ieeeconf}

\IEEEoverridecommandlockouts

\usepackage{amsmath,amssymb,pstricks,color,dsfont}
\usepackage{graphicx}
\usepackage{psfrag,graphicx,epsfig}
\usepackage{algorithm,algorithmic,xspace}

\newcommand{\nnum}{\nonumber}

\newcommand{\EQ}{\begin{eqnarray}}
\newcommand{\EN}{\end{eqnarray}}
\newcommand{\EQQ}{\begin{eqnarray*}}
\newcommand{\ENN}{\end{eqnarray*}}

\newcommand{\bremark}{\begin{remark} \begin{rm} }
\newcommand{\eremark}{ \end{rm} \rule{1mm}{2mm}
\end{remark} }
\newcommand{\btheorem}{\begin{theorem} \begin{rm} }
\newcommand{\etheorem}{ \end{rm} \rule{1mm}{2mm}
\end{theorem} }
\newcommand{\blemma}{\begin{lemma} \begin{rm} }
\newcommand{\elemma}{ \end{rm} \rule{1mm}{2mm}
\end{lemma} }
\newcommand{\bcorollary}{\begin{corollary} \begin{rm} }
\newcommand{\ecorollary}{ \end{rm} \rule{1mm}{2mm}
\end{corollary} }
\newcommand{\bdefinition}{\begin{definition}\begin{rm} }
\newcommand{\edefinition}{ \end{rm} \rule{1mm}{2mm}
\end{definition} }
\newcommand{\bproposition}{\begin{proposition} \begin{rm} }
\newcommand{\eproposition}{ \end{rm} \rule{1mm}{2mm}
\end{proposition} }
\newcommand{\bexample}{\begin{example} \begin{rm} }
\newcommand{\eexample}{ \end{rm} \rule{1mm}{2mm}
\end{example} }
\newcommand{\basm}{\begin{assumption} \begin{rm}}
\newcommand{\easm}{\end{rm} 
\end{assumption}}

\newcommand{\real}{\mathds{R}}

\newcommand{\CC}{\mathcal{C}}

\newcommand{\EE}{\mathcal{E}}

\newcommand{\GG}{\mathcal{G}}
\newcommand{\HH}{\mathcal{H}}

\newcommand{\LL}{\mathcal{L}}
\newcommand{\NN}{\mathcal{N}}

\newcommand{\dist}{\operatorname{dist}}

\newcommand{\diag}[1]{\operatorname{diag}(#1)}

\newtheorem{theorem}{\bf Theorem}[section]
\newtheorem{lemma}{\bf Lemma}[section]
\newtheorem{definition}{\bf Definition}[section]
\newtheorem{remark}{\bf Remark}[section]
\newtheorem{corollary}{\bf Corollary}[section]
\newtheorem{proposition}{\bf Proposition}[section]
\newtheorem{example}{\bf Example}[section]
\newtheorem{assumption}{\bf Assumption}[section]

\newcommand\oprocendsymbol{\hbox{$\bullet$}}
\newcommand\oprocend{\relax\ifmmode\else\unskip\hfill\fi\oprocendsymbol}

\date{}

\begin{document}

\title{Real-time game theoretic coordination of competitive mobility-on-demand systems}

\author{Minghui Zhu and Emilio Frazzoli \thanks{The authors are
   with Laboratory for Information and Decision Systems, Massachusetts Institute of Technology, 77 Massachusetts Avenue, Cambridge, MA 02139, {\tt\small \{mhzhu,frazzoli\}@mit.edu}. This research was supported in part by the Future Urban Mobility project of the Singapore-MIT Alliance for Research and Technology (SMART) Center, with funding from Singapore's National Research Foundation.}}

\maketitle

\begin{abstract} This paper considers competitive mobility-on-demand systems where a group of vehicle sharing companies provide pickup-delivery service in populated areas. The companies, on one hand, want to collectively regulate the traffic of the user queueing network, and on the other hand, aim to maximize their own net profit at each time instant by increasing the user delivery and reducing the transition of empty vehicles. We formulate the strategic interconnection among the companies as a real-time game theoretic coordination problem. We propose an algorithm to achieve vehicle balance and practical regulation of the user queueing network. We quantify the relation between the regulation error and the system parameters (e.g., the maximum variation of the user arrival rates).
\end{abstract}

\section{Introduction}

Private automobiles are not a sustainable solution to personal mobility given their drawbacks of energy inefficiency, high greenhouse gas emissions and induced traffic congestion. The report~\cite{Schrank.Lomax.Eisele:11} shows that in 2010, traffic congestion caused an annual cost of $101$ billion, and drivers spent $4.8$ billion hours in traffic in United States. Mobility-on-demand (MoD) systems represent a promising paradigm for future urban mobility. In particular, MoD systems are one-way vehicle-sharing systems where vehicle-sharing companies provide sharing vehicles at stations in a geographic region of interest, and users drive or are driven by the vehicles from a pickup location to a drop-off location. Several pilot programs have empirically demonstrated that MoD systems are efficient in addressing the drawbacks of private automobiles. In MoD system, the arrivals and departures of users are uncorrelated, so it is important to real-time reallocate the vehicles to match the dynamic and spatial demands. In this paper, we focus on competitive MoD systems where multiple service suppliers compete with one another to maximize their own profits. The paper~\cite{Pavone.Smith.Frazzoli.Rus:12} instead studies the scenario where there is a single service supplier.

\emph{Literature review.} Networked resource allocation among competing users has been extensively studied in the context of Game Theory. In~\cite{Altman.Basar:98}, the authors exploit differential game theory to derive Nash equilibrium controllers for multiple self-interested users to regulate the traffic of a single queue. However, the approach in~\cite{Altman.Basar:98} is not applicable to our problem due to; e.g, the additional dynamics of vehicle queues, the nonlinearity and non-smoothness of dynamic systems and the presence of state and input constraints. Static games have also been widely used to synthesize decentralized schemes for resource allocation, and a necessarily incomplete reference list includes~\cite{Alpcan.Basar:05,Marden.Wierman:08,Pan.Pavel:09,Yin.Shanbhag.Mehta:11,Zhu.Martinez:12}. Another relevant problem is demand response in the emerging smart grid where customers manage their electricity consumption in response to supply conditions. Some references on the regard include~\cite{Chen.Li.Jiang.Low:12,Ma.Callaway.Hiskens:12}.

Our problem is also related to (open-loop) optimization and games in dynamic environments. In~\cite{Cavalcante.Rogers.Jennings.Yamada:11}, the authors study the problem of seeking the common global optimum of a sequence of time-varying functions. The papers~\cite{Chen.Lau:12,Chen.Lau.Chen:11} investigate resource allocation of communication systems over time-varying fading channels. The online convex optimization and games have been considered in the papers~\cite{Gordon.Greenwald.Marks:08,Zinkevich:03}.

Another set of papers relevant to our work is concerned with generalized Nash games. This class of continuous games are first formulated in~\cite{Arrow.Debreu:54}. Since then, a great effort has been dedicated to investigating the existence and structural properties of generalized Nash equilibria. An incomplete reference list includes the recent survey paper~\cite{Facchinei.Kanzow:07} and~\cite{Basar.Olsder:82,Facchinei.Pang:03,Rosen:65}. There have been several algorithms proposed to compute generalized Nash equilibria, including ODE-based methods~\cite{Rosen:65}, nonlinear Gauss-Seidel-type approaches~\cite{Pang.Scutari.Facchinei.Wang:08}, iterative primal-dual Tikhonov schemes~\cite{Yin.Shanbhag.Mehta:11} and best-response dynamics~\cite{Palomar.Eldar:10}. In our recent paper~\cite{Zhu.Frazzoli:12}, we consider distributed computation of generalized Nash equilibria over unreliable communication networks.

\emph{Contributions.} In this paper, we present a model of competitive MoD systems and formulate the problem of real-time game theoretic coordination among multiple players (i.e., vehicle sharing companies). In particular, each player wants to collectively regulate the traffic of the user queueing network through delivering the users to their destinations. On the other hand, each player aims to maximize his net profit at each time instant by increasing the user delivery and reducing the transition of empty vehicles. We propose an algorithm to achieve vehicle balance and practical regulation of the user queueing network. The closed-loop system consists of a feedback connection of the cyber and physical layers: in the cyber layer, the players seek instantaneous Nash equilibrium in a distributed fashion; the intermediate estimates of Nash equilibrium are employed to control the physical queueing networks after a proper projection; the states of the queueing networks are injected into the cyber layer to keep track of Nash equilibrium. We quantify the relation between the regulation error and the system parameters (i.e., the maximum variation of the user arrival rates). For ease of presentation, the notations of Sections~\ref{sec:preliminaries} and~\ref{sec:main} will be introduced and summarized in the Appendix.


\section{Problem formulation}\label{sec:formulation}

In this section, we will provide a model for competitive MoD systems and introduce the problem formulation considered in the paper. Basic notations used in this section are summarized in Table~\ref{ta:basic}.

\begin{table}[ht]
\caption{Basic notations}
\medskip
\centering
\begin{tabular}{|c|l|}
  \hline
  $c_{\kappa}(t)$ & \text{user arrival rate at station~$\kappa$}
\tabularnewline
  \hline
  $v^{[i]}_{\kappa}(t)$ & \text{number of vehicles of player~$i$ at station~$\kappa$}
\tabularnewline
  \hline
  $\beta^{[i]}_{\kappa}(t)$ & \text{delivery rate of player~$i$ at station~$\kappa$}
\tabularnewline
  \hline
  $\alpha^{[i]}_{\kappa\kappa'}(t)$ & \text{transfer rate of empty vehicles of player~$i$}
\tabularnewline
  \hline
  $Q_{\kappa}(t)$ & \text{queue length of station~$\kappa$}
\tabularnewline
  \hline
  $u_{\kappa}(t)$ & \text{controller of station~$\kappa$}
\tabularnewline
  \hline
  $\textbf{1}$ & \text{indicator function}
\tabularnewline
  \hline
  $\mathcal{B}_i$ & \text{profit function of player~$i$}
\tabularnewline
  \hline
  $\mathcal{C}_i$ & \text{cost function of player~$i$}
\tabularnewline
  \hline
\end{tabular}\label{ta:basic}
\end{table}

\subsection{Model}

A competitive MoD system consists of three interconnected networks: the user queueing network, the vehicle queueing network and the player network. Figure~\ref{fig_system} shows the architecture of the system.

\begin{figure}[h]
  \centering
  \includegraphics[width=.8\linewidth]{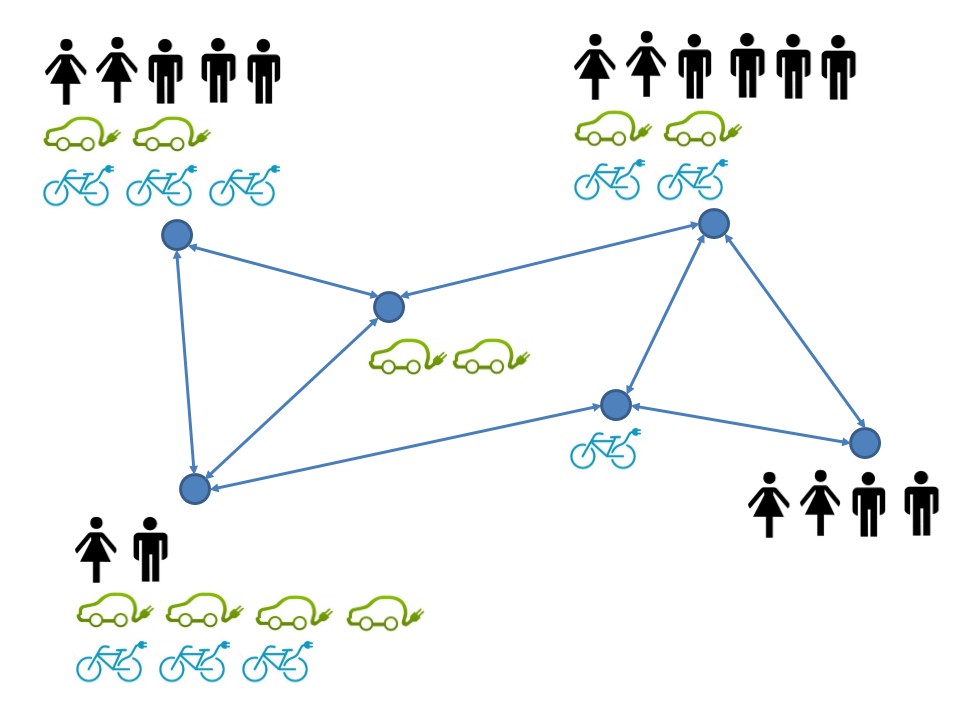}
  \caption{A competitive MoD system where the companies of shared electrical vehicles and bicycles serve the area} \label{fig_system}
\end{figure}

\subsubsection{The user queueing network}

There is a set of stations, say $\mathbb{S}$, in a spatial area of interest, and the interconnection of the stations is characterized by the graph $\GG_{\mathbb{S}} \triangleq \{\mathbb{S},\EE_{\mathbb{S}}\}$ where $(\kappa,\kappa')\in\EE_{\mathbb{S}}\setminus\diag{\mathbb{S}}$ if and only if the users at station~$\kappa$ can be delivered to station~$\kappa'$. The graph $\GG_{\mathbb{S}}$ is fixed, undirected and connected. Denote by $\NN_{\kappa} \triangleq \{\kappa'\in\mathbb{S} \; | \; (\kappa,\kappa')\in\EE_{\mathbb{S}}\}$ the set of neighboring stations of station~$\kappa$.

Users arrive at station $\kappa\in\mathbb{S}$ in a dynamic fashion. Let $c_{\kappa}(t)\in\real_{\geq0}$ be the user arrival rate at station $\kappa$ at time~$t$, and its temporal evolution is governed by the following ordinary differential equation: \begin{align} \dot{c}_{\kappa}(t) = h_{\kappa}(c_{\kappa}(t),Q_{\kappa}(t),t).\label{e11}
\end{align} In~\eqref{e11}, $Q_{\kappa}(t)$ is the queue length of station~$\kappa$ and will be defined later. The function $h_{\kappa} : \real^3_{\geq0}\rightarrow \real_{\geq0}$ is locally Lipschitz in $(c_{\kappa}(t),Q_{\kappa}(t))$ and piecewise continuous in $t$. Let $a_{\kappa\kappa'}(t)\in[0,1]$ be the fraction of users who arrive at station~$\kappa$ at time~$t$ and want to reach station~$\kappa'\neq\kappa$. Thus $\sum_{\kappa'\in\NN_{\kappa}}a_{\kappa\kappa'}(t) = 1$. We assume that the fraction $a_{\kappa\kappa'}(t)$ is fixed; i.e., $a_{\kappa\kappa'}(t) = a_{\kappa\kappa'}$ for $t\geq0$.

A queue is associated with each station $\kappa \in \mathbb{S}$, and the arrived users wait for the delivery in the queue. Let $Q_{\kappa}(t)\in\real_{\geq0}$ be the queue length of station $\kappa \in \mathbb{S}$ at time~$t$, and it dynamics is given by:
\begin{align} \dot{Q}_{\kappa}(t) = (c_{\kappa}(t) - u_{\kappa}(t))\mathbf{1}_{[Q_{\kappa}(t)\geq0]},\label{e10}
\end{align} where the initial state $Q_{\kappa}(0) > 0$, and the quantity $u_{\kappa}(t) = \sum_{i\in V}\beta^{[i]}_{\kappa}(t)\in\real_{\geq0}$ where $\beta^{[i]}_{\kappa}(t)$ is the delivery rate of player~$i$ at station~$\kappa$ and $V \triangleq \{1,\cdots,N\}$ is the set of players explained later.

Denote the vectors $Q \triangleq (Q_{\kappa})_{\kappa\in\mathbb{S}}$, $c \triangleq (c_{\kappa})_{\kappa\in\mathbb{S}}$, $\xi\triangleq (Q,c)$ and $u \triangleq (u_{\kappa})_{\kappa\in\mathbb{S}}$. We then rewrite~\eqref{e11} and~\eqref{e10} into the following compact form: \begin{align}\dot{Q}(t) &= h_Q(c(t),u(t)),\label{e18}\\
\dot{c}(t) &= h_c(c(t),Q(t),t).\label{e14}\end{align}
and then, \begin{align}\dot{\xi}(t) &= h_{\xi}(\xi(t),u(t),t).\label{e15}\end{align}

\subsubsection{The vehicle queueing network}

There is a group of players $V \triangleq \{1,\cdots,N\}$. Each player is a vehicle-sharing company, and he provides the service of delivering the users on the graph $\mathcal{G}_{\mathbb{S}}$. Let $v^{[i]}_{\kappa}(t)\in\real_{\geq0}$ be the number of vehicles of player~$i$ stored at station~$\kappa$ at time~$t$. If $v^{[i]}_{\kappa}(t) > 0$, then player~$i$ is able to deliver the users leaving station~$\kappa$ at a rate $\beta^{[i]}_{\kappa}(t)\in[a,\beta^{[i]}_{\max}-a]$ with $0<a<\frac{1}{2}\beta^{[i]}_{\max}$\footnote{The quantity $a>0$ could be chosen arbitrarily small.}; otherwise, player~$i$ cannot deliver any user; i.e., $\beta^{[i]}_{\kappa}(t) = 0$. In order to avoid $v^{[i]}_{\kappa}(t)$ becoming zero, each player~$i$ needs to reallocate his empty vehicles. Let $\alpha^{[i]}_{\kappa\kappa'}(t)\in[a,\alpha^{[i]}_{\max}-a]$ with $0<a<\frac{1}{2}\alpha^{[i]}_{\max}$ be the rate that the empty vehicles of player~$i$ are transferred from station~$\kappa$ to station~$\kappa'$ at time~$t$. The dynamics of $v^{[i]}_{\kappa}(t)$ is based on the mass-conservation law and given by: \begin{align}\dot{v}^{[i]}_{\kappa}(t) &= \{-\beta^{[i]}_{\kappa}(t) + \sum_{\kappa'\in\NN_{\kappa}}a_{\kappa'\kappa}\beta^{[i]}_{\kappa'}(t)\nnum\\
&- \sum_{\kappa'\in\NN_{\kappa}}\alpha^{[i]}_{\kappa\kappa'}(t) + \sum_{\kappa'\in\NN_{\kappa}}\alpha^{[i]}_{\kappa'\kappa}(t)\}\textbf{1}_{[v^{[i]}_{\kappa}(t)>0]}\nnum\\
&+ \{\sum_{\kappa'\in\NN_{\kappa}}a_{\kappa'\kappa}\beta^{[i]}_{\kappa'}(t) + \sum_{\kappa'\in\NN_{\kappa}}\alpha^{[i]}_{\kappa'\kappa}(t)\}\textbf{1}_{[v^{[i]}_{\kappa}(t)=0]},\label{e6}\end{align} with the initial state $v^{[i]}_{\kappa}(0) > 0$. It is easy to verify that $v^{[i]}_{\kappa}(t)\geq0$ for all $t\geq0$.

\subsubsection{The player network}

Each player in $V$ has three partially conflicting objectives: the first one is to collectively regulate the queue length $Q_{\kappa}(t)$ to the desired level $\bar{Q}_{\kappa}\in\real_{>0}$, the second one is to maintain $v^{[i]}_{\kappa}(t)$ to be strictly positive, and the third one is to maximize his own net profits. In the sequel, we will explain each objective in more detail.

Firstly, players aim to collectively regulate the queue length $Q_{\kappa}(t)$ to the desired level $\bar{Q}_{\kappa}$. For the time being, we assume that there exists a smooth controller $U(\xi(t))\triangleq (U_{\kappa}(\xi_{\kappa}(t)))_{\kappa\in \mathbb{S}}$ which is able to achieve the queue regulation. Hence, players share a common goal of enforcing the controllers $u_{\kappa}(t) = U_{\kappa}(\xi_{\kappa}(t))$ for all $\kappa\in \mathbb{S}$.

Secondly, each player~$i$ wants to maintain $v^{[i]}_{\kappa}(t) > 0$ in order to sustain a non-trivial service rate $\beta^{[i]}_{\kappa}(t) > 0$. Since $v^{[i]}_{\kappa}(0) > 0$, player~$i$ can achieve this goal through simply keeping $v^{[i]}_{\kappa}(t)$ as a constant; i.e., enforcing the hard constraint $\dot{v}^{[i]}_{\kappa}(t) = 0$ all the time.

Thirdly, each player is self-interested and desires to maximize his net profit at each time instant. In particular, each player~$i$ is able to make a profit from the delivery service, and the profit is modeled by $\mathcal{B}_i(\beta^{[i]}_{\kappa}(t))$ where the function $\mathcal{B}_i : \real_{\geq0}\rightarrow\real_{\geq0}$ is smooth and strongly concave with constant $\rho_i > 0$. On the one hand, the transfer of empty vehicles is costly, and the expense is modeled by $\mathcal{C}_i(\alpha^{[i]}_{\kappa\kappa'}(t))$ where $\mathcal{C}_i : \real_{\geq0}\rightarrow\real_{\geq0}$ is smooth and strongly convex with constant $\rho'_i > 0$. The net cost of player $i$ at time~$t$ is abstracted by  $\sum_{(\kappa,\kappa')\in\EE_{\mathbb{S}}}\CC_i(\alpha^{[i]}_{\kappa\kappa'}(t)) - \sum_{\kappa\in\mathbb{S}}\mathcal{B}_i(\beta^{[i]}_{\kappa}(t))$.

The decision vector of player~$i$ at time~$t$ is given by $z^{[i]}(t)$ which is the collection of $\alpha^{[i]}(t)\triangleq(\alpha^{[i]}_{\kappa\kappa'}(t))_{(\kappa,\kappa')\in\EE_{\mathbb{S}}}$ and $\beta^{[i]}(t)\triangleq(\beta^{[i]}_{\kappa}(t))_{\kappa\in\mathbb{S}}$. The above three interests of player~$i$ at time~$t$ are compactly expressed by the following convex program parameterized by the vector $U(\xi(t))$:
\begin{align}&\min_{z^{[i]}(t)\in \real^{n_i}}\;\; \sum_{(\kappa,\kappa')\in\EE_{\mathbb{S}}}\CC_i(\alpha^{[i]}_{\kappa\kappa'}(t)) - \sum_{\kappa\in\mathbb{S}}\mathcal{B}_i(\beta^{[i]}_{\kappa}(t)),\nnum\\
&{\rm s.t.} \quad \sum_{i\in V}\beta^{[i]}_{\kappa}(t) = U_{\kappa}(\xi_{\kappa}(t)),\quad \kappa\in\mathbb{S},\nnum\\
&\quad\quad -\beta^{[i]}_{\kappa}(t) + \sum_{\kappa'\in\NN_{\kappa}}a_{\kappa'\kappa}\beta^{[i]}_{\kappa'}(t) - \sum_{\kappa'\in\NN_{\kappa}}\alpha^{[i]}_{\kappa\kappa'}(t)\nnum\\
&\quad\quad + \sum_{\kappa'\in\NN_{\kappa}}\alpha^{[i]}_{\kappa'\kappa}(t) = 0,\quad \kappa\in\mathbb{S},\nnum\\
&\quad\quad z^{[i]}(t) \in Z_i,\label{e12}\end{align} where the dimension $n_i = |\mathbb{S}| + 2|\mathcal{E}_{\mathbb{S}}|$ and the set $Z_i$ is defined as: $Z_i \triangleq \{z^{[i]}\;|\;\beta^{[i]}_{\kappa}\in[a,\beta^{[i]}_{\max}-a],\quad \kappa\in\mathbb{S},\quad\alpha^{[i]}_{\kappa\kappa'}\in[a,\alpha^{[i]}_{\max}-a],\quad (\kappa,\kappa')\in\EE_{\mathbb{S}}\}$. In~\eqref{e12}, the decisions of the players are coupled via the constraint $\sum_{i\in V}\beta^{[i]}_{\kappa}(t) = U_{\kappa}(\xi_{\kappa}(t))$ which represents the common goal of the queue regulation. Other components in~\eqref{e12} are instead separable.

By using $v(x)=0$ if and only if $v(x)\leq0$ and $v(x)\geq0$, we rewrite the parametric convex program~\eqref{e12} into the following compact form: \begin{align}&\min_{z^{[i]}(t)\in \real^{n_i}}f_i(z^{[i]}(t)),\nnum\\
&{\rm s.t.} \quad G(\beta^{[i]}(t),\beta^{[-i]}(t),U(\xi(t))) \leq 0,\nnum\\
&\quad\quad\; h^{[i]}(z^{[i]}(t)) \leq 0,\quad z^{[i]}(t) \in Z_i,\label{obj-fun}\end{align} where $G : \real^{N |\mathbb{S}|}\rightarrow\real^m$ ($m = 2|\mathbb{S}|$) and $h^{[i]} : \real^{n_i}\rightarrow\real^{p}$ ($p = 2|\mathbb{S}|$) are affine functions. The components of $G$ and $h^{[i]}$ are asymmetric; i.e., $G_{\ell} = - G_{\ell+|\mathbb{S}|}$ and $h^{[i]}_{\ell} = - h^{[i]}_{\ell+|\mathbb{S}|}$ for $1 \leq \ell \leq |\mathbb{S}|$. The collection of~\eqref{obj-fun} will be referred to as the CVX game parameterized by $U(\xi(t))$, and its solution, Nash equilibrium, is defined as follows: \begin{definition} For the CVX game parameterized by $U(\xi(t))$, the state $\tilde{z}(t)\in Z\triangleq \prod_{i\in V}Z_i$ is a Nash equilibrium if and only if:
\begin{enumerate}
\item[(1)] $G(\tilde{\beta}(t),U(\xi(t)))\leq0$ and $h^{[i]}(\tilde{z}^{[i]}(t))\leq0$;
\item[(2)] for any $z^{[i]}\in Z_i$ with $G(\beta^{[i]},\tilde{\beta}^{[-i]}(t),U(\xi(t)))\leq0$ and $h^{[i]}(z^{[i]})\leq0$, it holds that $f_i(\tilde{z}^{[i]}(t)) \leq f_i(z^{[i]})$.
\end{enumerate}\label{def3}\end{definition}

The set of Nash equilibria is denoted by $\mathbb{X}_{\rm C}(\xi(t))$. Since $f_i$ is strongly convex and separable, the map of partial gradients is strongly monotone, and thus $\mathbb{X}_{\rm C}(\xi(t))$ is non-empty; e.g., in~\cite{Facchinei.Kanzow:07}.

\subsection{Our objective}

At each time instant~$t$, the players aim to solve the CVX game parameterized by the control command $U(\xi(t))$, and implement a Nash equilibrium in $\mathbb{X}_{\rm C}(\xi(t))$. This procedure is repeated at the next time instant by shifting the time horizon forward. We term the collection of these finite-horizon games over the infinite horizon as a \emph{real-time game}. In this paper, we will design an algorithm to update $z(t)$ such that real-time game theoretic coordination is achieved; that is, \begin{align}&\lim_{t\rightarrow+\infty}\dist(z(t),\mathbb{X}_{\rm C}(\xi(t))) = 0,\nnum\\
&\lim_{t\rightarrow+\infty}Q_{\kappa}(t) = \bar{Q}_{\kappa},\quad \kappa\in \mathbb{S},\nnum\\
&v^{[i]}_{\kappa}(t) = v^{[i]}_{\kappa}(0),\quad \kappa\in \mathbb{S},\quad i\in V,\quad t\geq0.\label{e20}\end{align}

\begin{remark} In contrast to~\cite{Altman.Basar:98,Basar.Olsder:82}, our real-time game theoretic coordination formulation relaxes the computation of infinite-horizon Nash equilibrium. Instead, our formulation aims to real-time seek the collection of instantaneous Nash equilibrium. By Lemma~\ref{lem3}, one can see that if $z(t)$ asymptotically achieves $\mathbb{X}_{\rm C}(\xi(t))$, then the infinite-horizon average performance of $z(t)$ is identical to that of $\mathbb{X}_{\rm C}(\xi(t))$. More importantly, the formulation allows us to handle constrained discontinuous dynamic systems and relax the \emph{a priori} information of the arrival rates over the infinite horizon.

Our game formulation is partially motivated by receding-horizon control or model predictive control; e.g., in~\cite{Mayne.Rawlins.Rao.Scokaert:00}, whose control laws are based on solving a sequence of finite horizon optimal control problems. Our game formulation is also partially inspired by optimization and games in dynamic environments; e.g., in~\cite{Cavalcante.Rogers.Jennings.Yamada:11,Chen.Lau:12,Gordon.Greenwald.Marks:08}. However, this set of papers only consider open-loop decision making.\oprocend\label{rem5}
\end{remark}

\subsection{Assumptions}

%
%

Let $\beta_{\max}\triangleq\sum_{i\in V}\beta^{[i]}_{\max}$, $\alpha_{\max}\triangleq\sum_{i\in V}\alpha^{[i]}_{\max}$. In the remainder of this paper, we suppose that the following set of assumptions hold.

\begin{assumption} It holds that $\beta^{[i]}_{\max}$ are identical for all $i\in V$ and $c_{\kappa}(t)\in[c_{\min},c_{\max}]$ for all $t\geq0$ and $\kappa\in\mathbb{S}$. In addition, $2Na \leq c_{\min} < c_{\max} < \beta_{\max}-Na$.\label{asm1}
\end{assumption}

\begin{assumption} There is $\delta_c>0$ such that $\|\dot{c}_{\kappa}(t)\|\leq\delta_c$ for all $\kappa$ and $t\geq0$.\label{asm7}
\end{assumption}

\begin{assumption} For any $\beta^{[i]}$ with $\beta^{[i]}_{\kappa}\in[a,\beta^{[i]}_{\max}-a]$, there is $\alpha^{[i]}$ such that $\alpha^{[i]}_{\kappa\kappa'}\in[a,\alpha^{[i]}_{\max}-a]$ and the following holds for $\kappa\in\mathbb{S}$: \begin{align*}-\beta^{[i]}_{\kappa} + \sum_{\kappa'\in\NN_{\kappa}}a_{\kappa'\kappa}\beta^{[i]}_{\kappa'} - \sum_{\kappa'\in\NN_{\kappa}}\alpha^{[i]}_{\kappa\kappa'}
+ \sum_{\kappa'\in\NN_{\kappa}}\alpha^{[i]}_{\kappa'\kappa} = 0.\end{align*}\label{asm4}
\end{assumption}

Assumption~\ref{asm1} requires that the maximum delivery rate is larger than the maximum arrival rate. This assumption is necessary for the queue stabilization. Assumption~\ref{asm7} means that the variations of the arrival rates are bounded. The combination of Assumptions~\ref{asm1} and~\ref{asm7} implies that there is $\delta_{\xi}>0$ such that $\|\dot{\xi}(t)\|\leq \delta_{\xi}$ for all $t\geq0$. Assumption~\ref{asm4} implies that given any feasible delivery vector $\beta^{[i]}$, each player~$i$ is able to maintain the vehicle balance at different stations. Since $\mathcal{G}_{\mathbb{S}}$ is undirected, this assumption requires that $\alpha^{[i]}_{\max}$ is large enough in comparison with $\beta^{[i]}_{\max}$.

\section{Preliminaries}\label{sec:preliminaries}

In the sequel, we will first introduce an approximation of the CVX game parameterized by $\zeta(t)\triangleq U(\xi(t))$, namely, the regularized game. In order to simplify the notations, we will drop the dependency of $\zeta(t)$ on time~$t$. We will then characterize the distance between the CVX game and the regularized game. After this, we will perform sensitivity analysis on the regularized game.

\subsection{The existence of smooth controllers}

With Assumption~\ref{asm1}, we will show that the regulation of user queues can be achieved via the following smooth controller: \begin{align}&u_{\kappa}(t) = U_{\kappa}(\xi_{\kappa}(t)) = c_{\kappa}(t) - \hat{U}_{\kappa}(Q_{\kappa}(t))\nnum\\
&\triangleq c_{\kappa}(t)-\frac{c_{\min}}{2}
+\frac{\beta_{\max}-c_{\max}
+\frac{c_{\min}}{2}-Na}{1+2\frac{\beta_{\max}-c_{\max}-Na}{c_{\min}}e^{-(Q_{\kappa}(t)-\bar{Q}_{\kappa})}}.\label{e21}\end{align} Towards this end, it is easy to verify that $\hat{U}_{\kappa}(\bar{Q}_{\kappa}) = c_{\kappa}(t)$ and $U_{\kappa}(\xi_{\kappa}(t))\in[\frac{c_{\min}}{2},\beta_{\max}-Na]\subseteq[Na,\beta_{\max}-Na]$ by utilizing the monotonicity of the functions in $\hat{U}_{\kappa}$. Hence, the controller $U_{\kappa}(\xi_{\kappa}(t))$ is realizable for the players. Furthermore, the Lie derivative of the regulation error $\frac{1}{2}(Q_{\kappa}(t)-\bar{Q}_{\kappa})^2$ along~\eqref{e10} for $Q_{\kappa}(t)\geq0$ is given by: \begin{align}&\frac{1}{2}\frac{d}{dt}(Q_{\kappa}(t)-\bar{Q}_{\kappa})^2= (Q_{\kappa}(t)-\bar{Q}_{\kappa})\hat{U}_{\kappa}(Q_{\kappa}(t))\nnum\\
&=\frac{\beta_{\max}-c_{\max}-Na}{1+2\frac{\beta_{\max}-c_{\max}-Na}{c_{\min}}e^{-(Q_{\kappa}(t)-\bar{Q}_{\kappa})}}\nnum\\
&\times(Q_{\kappa}(t)-\bar{Q}_{\kappa})(e^{-(Q_{\kappa}(t)-\bar{Q}_{\kappa})}-1).\label{e17}\end{align}
So, $(Q_{\kappa}(t)-\bar{Q}_{\kappa})\hat{U}_{\kappa}(Q_{\kappa}(t)) < 0$ for $Q_{\kappa}(t)-\bar{Q}_{\kappa}\neq0$ and $Q_{\kappa}(t)\geq0$. Hence, $U_{\kappa}(\xi_{\kappa}(t))$ is able to regulate $Q_{\kappa}(t)$ to $\bar{Q}_{\kappa}$ from any initial state $Q_{\kappa}(0)>0$. This controller will be used in the remainder of the paper. In the sequel, we will find uniform upper bounds on $\|\frac{d U(\xi)}{d \xi}\|$ and $\|\frac{d^2 U(\xi)}{d \xi^2}\|$.
\begin{lemma}The following holds for all $\xi\geq0$: \begin{align}&\|\frac{d U(\xi)}{d \xi}\|\leq D_U^{(1)},\quad\|\frac{d^2 U(\xi)}{d \xi^2}\|\leq D_U^{(2)}.\label{e74}\end{align}\label{lem2}
\end{lemma}

\begin{proof} Notice that \begin{align}\frac{d U_{\kappa}(\xi_{\kappa})}{d Q_{\kappa}} &= \frac{\beta_{\max}-c_{\max}+\frac{c_{\min}}{2}-Na}{(1+2\frac{\beta_{\max}-c_{\max}-Na}{c_{\min}}e^{-(Q_{\kappa}-\bar{Q}_{\kappa})})^2}\nnum\\
&\times \frac{(-2)(\beta_{\max}-c_{\max}-Na)}{c_{\min}}e^{-(Q_{\kappa}-\bar{Q}_{\kappa})},\nnum\\
\frac{d^2 U_{\kappa}(\xi_{\kappa})}{d Q_{\kappa}^2} &= \frac{(\beta_{\max}-c_{\max}+\frac{c_{\min}}{2}-Na)\frac{2(\beta_{\max}-c_{\max}-Na)}{c_{\min}}}{(1+2\frac{\beta_{\max}-c_{\max}-Na}{c_{\min}}e^{-(Q_{\kappa}-\bar{Q}_{\kappa})})^4}\nnum\\
&\times(1-(\frac{2(\beta_{\max}-c_{\max}-Na)}{c_{\min}}e^{-(Q_{\kappa}-\bar{Q}_{\kappa})})^2)\nnum\\
&\times e^{-(Q_{\kappa}-\bar{Q}_{\kappa})},\nnum\\
\frac{d U_{\kappa}(\xi_{\kappa})}{d c_{\kappa}} &= 1.\nnum\end{align} The above relations in conjunction with $e^{-(Q_{\kappa}-\bar{Q}_{\kappa})}\in[0,e^{\bar{Q}_{\kappa}}]$ establish the desired bounds.\end{proof}

\subsection{The regularized game}

\subsubsection{Regularized Lagrangian functions}

To relax the constraints of $G(\beta,\zeta)\leq0$ and $h^{[i]}(z^{[i]})\leq0$, we define the following \emph{regularized} Lagrangian function for player~$i$: \begin{align}&\LL_i(z,\mu,\lambda^{[i]},\zeta) = f_i(z^{[i]}) + \langle\mu, G(\beta,\zeta)\rangle + \langle\lambda^{[i]},h^{[i]}(z^{[i]})\rangle\nnum\\
&-\tau\sum_{\kappa\in \mathbb{S}}\big(\psi(\beta^{[i]}_{\kappa})+\psi(\beta^{[i]}_{\max}-\beta^{[i]}_{\kappa})\big)\nnum\\ &-\tau\sum_{(\kappa,\kappa')\in \mathcal{E}_{\mathbb{S}}}\big(\psi(\alpha^{[i]}_{\kappa\kappa'})
+\psi(\alpha^{[i]}_{\max}-\alpha^{[i]}_{\kappa\kappa'})\big)\nnum\\
&- \frac{\epsilon}{2}\|\mu\|^2 - \frac{\epsilon}{2}\|\lambda^{[i]}\|^2 + \tau \sum_{\ell=1}^m\psi(\mu_{\ell}) + \tau \sum_{\ell=1}^{p}\psi(\lambda^{[i]}_{\ell}),\label{e16}\end{align} with $\epsilon > 0$, $\tau > 0$ and $\mu\in\real^m$ and $\lambda^{[i]}\in\real^{p}$ are dual multipliers. The function $\psi$ is the logarithmic barrier function and defined as follows:
\begin{align*}\left\{\begin{array}{ll}
                \psi(s) = \log (\frac{s}{a}), & s > 0, \\
                \psi(s) = -\infty, & s \leq 0.
              \end{array}\right.
\end{align*} Note that $\psi$ is concave and monotonically increasing over $\real_{>0}$. In $\LL_i$, the hard constraints $\mu_{\ell}\geq0$, $\lambda^{[i]}_{\ell}\geq0$ and $z^{[i]}\in Z_i$ are relaxed by those defined via the logarithmic function. In addition, the terms associated with $\epsilon$ play a role of regularization as shown in Lemma~\ref{lem1}.

We then introduce a set of dual players $\{0\}\cup V_m\triangleq \{1,\cdots,N\}$, and $\mu$ is the decision vector of dual player~$0$, and $\lambda^{[i]}$ is the decision vector of dual player~$i$. Each primal player~$i\in V$ aims to minimize $\LL_i$ over $z^{[i]}\in\real^{n_i}$. Each dual player~$i\in V_m$ desires to maximize $\LL_i$ over $\lambda^{[i]}\in\real^{p}$ and dual player~$0$ wants to maximize $\HH(z,\mu,\zeta) \triangleq \langle\mu, G(\beta,\zeta)\rangle - \frac{\epsilon}{2}\|\mu\|^2 + \tau \sum_{\ell=1}^m\psi(\mu_{\ell})$. This game is referred to as the \emph{regularized} game (RG game, for short) parameterized by $\zeta$ and the definition of its NEs is given as follows:

\begin{definition} The state $(\tilde{z},\tilde{\mu},\tilde{\lambda})\in\real^{n+m+p}$ is a Nash equilibrium of the RG game parameterized by $\zeta$ if and only if the following hold for each primal player $i\in V$: \begin{align*}\LL_i(\tilde{z},\tilde{\mu},\tilde{\lambda}^{[i]},\zeta) \leq \LL_i(z^{[i]},\tilde{z}^{[-i]},\tilde{\mu},\tilde{\lambda}^{[i]},\zeta),\quad \forall z^{[i]}\in\real^{n_i},\end{align*} the following hold for each dual player $i\in V_m$: \begin{align*}\LL_i(\tilde{z},\tilde{\mu},\lambda^{[i]},\zeta) \leq \LL_i(\tilde{z},\tilde{\mu},\tilde{\lambda}^{[i]},\zeta),\quad \forall \lambda^{[i]}\in\real^{p},\end{align*} and the following hold for dual player $0$: \begin{align*}\HH(\tilde{z},\mu,\zeta) \leq \HH(\tilde{z},\tilde{\mu},\zeta),\quad \forall \mu\in\real^m.\end{align*} \label{def2}
\end{definition}

The set of NEs of the RG game parameterized by $\zeta$ is denoted as $\mathbb{X}_{\rm RG}(\zeta)$. Since the logarithmic barrier function $\psi$ penalizes $\mu\notin\real^m_{\leq0}$ an infinite cost, thus it must be $\tilde{\mu}(\zeta) > 0$ for any NE $\tilde{\eta}(\zeta)\in\mathbb{X}_{\rm RG}(\zeta)$. Analogously, $\tilde{\lambda}^{[i]}(\zeta) > 0$, $\tilde{\beta}^{[i]}_{\kappa}(\zeta)\in(0,\beta^{[i]}_{\max})$ and $\tilde{\alpha}^{[i]}_{\kappa\kappa'}(\zeta)\in(0,\alpha^{[i]}_{\max})$ for any NE $\tilde{\eta}(\zeta)\in\mathbb{X}_{\rm RG}(\zeta)$.

\subsubsection{Convexity of the RG game}

Since $f_i$ is strongly convex in $z^{[i]}$ and $\psi$ is concave over $\real_{>0}$, then $\LL_i$ is strongly convex in $z^{[i]}\in \hat{Z}_i \triangleq \{z^{[i]}\in\real^{n_i}\;|\;\beta^{[i]}_{\kappa}\in[0,\beta^{[i]}_{\max}],\quad \kappa\in\mathbb{S},\quad\alpha^{[i]}_{\kappa\kappa'}\in[0,\alpha^{[i]}_{\max}],\quad (\kappa,\kappa')\in\EE_{\mathbb{S}}\}$ with constant $\min\{\rho_i,\rho_i'\}$. By introducing the quadratic perturbation of $\frac{\epsilon}{2}\|\lambda^{[i]}\|^2$, the function of $\LL_i(z,\mu,\cdot,\zeta)$ is strongly concave in $\lambda^{[i]}$ with constant $\epsilon$ over $\real^{p}_{>0}$. This can be verified via the following computation: \begin{align}\frac{d^2 \LL_i}{d (\lambda^{[i]}_{\ell})^2} = -\epsilon - \frac{\tau}{(\lambda^{[i]}_{\ell})^2} < -\epsilon.\label{e31}\end{align} Analogously, the function of $\HH(z,\mu)$ is strongly concave in $\mu$ with constant $\epsilon$ over $\real^m_{>0}$.

\subsubsection{Monotonicity of the RG game}

It is noted that all the functions involved in $\LL_i$ are smooth in $\hat{Z}\times\real^m_{>0}\times\real^{p}_{>0}$. We then define $\nabla_{z^{[i]}}\LL_i(z,\mu,\lambda^{[i]},\zeta) : \hat{Z}\times\real^m_{>0}\times\real^{p}_{>0}\rightarrow\real^{n_i}$ as the partial gradient of the function $\LL_i(\cdot,z^{[-i]},\mu,\lambda^{[i]},\zeta)$ at $z^{[i]}$. Other partial gradients can be defined in an analogous way. Let $\eta\triangleq (z,\mu,\lambda)$, and define the map $\nabla \Omega : \hat{Z}\times\real^m_{>0}\times\real^p_{>0}\rightarrow\real^{n+m+p}$ as partial gradients of the player's objective functions: \begin{align*}&\nabla\Omega(\eta,\zeta)\nnum\\ &\triangleq \big[\nabla_{z^{[1]}}\LL_1(z,\mu,\lambda^{[1]},\zeta)^T\dots
\nabla_{z^{[N]}}\LL_N(z,\mu,\lambda^{[N]},\zeta)^T\\
&\quad\;\;-\nabla_{\mu}\HH(z,\mu,\zeta)^T\\
&\quad-\nabla_{\lambda^{[1]}}\LL_1(z,\mu,\lambda^{[1]},\zeta)^T\dots
-\nabla_{\lambda^{[N]}}\LL_N(z,\mu,\lambda^{[N]},\zeta)^T\big]^T.\end{align*}


The following lemma shows that the quadratic perturbations of $\frac{\epsilon}{2}\|\mu\|^2$ and $\frac{\epsilon}{2}\|\lambda^{[i]}\|^2$ regularize the game map $\nabla\Omega$ to be strongly monotone over $\hat{Z}\times\real^{m+p}_{>0}$.

\begin{lemma} The regularized game map $\nabla\Omega(\eta,\zeta)$ is strongly monotone over $\hat{Z}\times\real^{m+p}_{>0}$ with constant $\rho_{\Omega} = \min\{\min_{i\in V}\{\rho_i,\rho_i'\},\epsilon\}$. In addition, there is a unique NE $\tilde{\eta}(\zeta)\in\mathbb{X}_{\rm RG}(\zeta)$.\label{lem1}
\end{lemma}

\begin{proof} Pick any pair of $\eta,\bar{\eta}\in \hat{Z}\times\real^{m+p}_{>0}$. Since $G$ and $h^{[i]}$ are affine, one can verify that
\begin{align}&\langle \nabla\Omega(\eta,\zeta)-\nabla\Omega(\bar{\eta},\zeta), \eta-\bar{\eta} \rangle\nnum\\
&\geq \sum_{i\in V}\langle\nabla_{z^{[i]}}f_i(z^{[i]})-\nabla_{z^{[i]}}
f_i(\bar{z}^{[i]}),z^{[i]}-\bar{z}^{[i]}\rangle\nnum\\
&+ \epsilon\|\mu-\bar{\mu}\|^2 + \epsilon\sum_{i\in V}\|\lambda^{[i]}-\bar{\lambda}^{[i]}\|^2 + B_1 + B_2 + B_3,\label{e47}\end{align}
where the terms of $B_1$, $B_2$ and $B_3$ are given by: \begin{align}B_1 &\triangleq -\tau\sum_{i\in V}\sum_{\kappa\in \mathbb{S}}(\frac{1}{\beta^{[i]}_{\kappa}}-\frac{1}{\bar{\beta}^{[i]}_{\kappa}})
(\beta^{[i]}_{\kappa}-\bar{\beta}^{[i]}_{\kappa})\nnum\\
&+\tau\sum_{i\in V}\sum_{\kappa\in \mathbb{S}}(\frac{1}{\beta^{[i]}_{\max}-\beta^{[i]}_{\kappa}}
-\frac{1}{\beta^{[i]}_{\max}-\bar{\beta}^{[i]}_{\kappa}})
(\beta^{[i]}_{\kappa}-\bar{\beta}^{[i]}_{\kappa}),\nnum\\
B_2 &\triangleq -\tau\sum_{i\in V}\sum_{(\kappa,\kappa')\in \mathcal{E}_{\mathbb{S}}}(\frac{1}{\alpha^{[i]}_{\kappa\kappa'}}
-\frac{1}{\bar{\alpha}^{[i]}_{\kappa\kappa'}})
(\alpha^{[i]}_{\kappa\kappa'}-\bar{\alpha}^{[i]}_{\kappa\kappa'})\nnum\\
&+\tau\sum_{i\in V}\sum_{(\kappa,\kappa')\in \mathcal{E}_{\mathbb{S}}}(\frac{1}{\alpha^{[i]}_{\max}-\alpha^{[i]}_{\kappa\kappa'}}
-\frac{1}{\alpha^{[i]}_{\max}-\bar{\alpha}^{[i]}_{\kappa\kappa'}})
(\alpha^{[i]}_{\kappa\kappa'}-\bar{\alpha}^{[i]}_{\kappa\kappa'}),\nnum\\
B_3 &\triangleq \tau \sum_{\ell=1}^m(\log(\mu_{\ell})-\log(\bar{\mu}_{\ell}))(\mu_{\ell}
-\bar{\mu}_{\ell})\nnum\\
&+ \tau \sum_{i\in V}\sum_{\ell=1}^{p}(\log(\lambda^{[i]}_{\ell})-\log(\bar{\lambda}^{[i]}_{\ell}))
(\lambda^{[i]}_{\ell}
-\bar{\lambda}^{[i]}_{\ell}).\nnum\end{align}

By using the monotonicity of functions in $B_1, B_2, B_3$, it is readily to verify that $B_1, B_2, B_3\geq0$. Apply the mean-value theorem for vector functions and $f_i$ is strongly convex with constant $\min\{\rho_i,\rho_i'\}$ to~\eqref{e47}. We then reach the desired result that $\nabla \Omega$ is strongly monotone over $\hat{Z}\times\real^{m+p}_{>0}$ with constant $\rho_{\Omega}$. The strong monotonicity of $\nabla\Omega$ ensures the existence and uniqueness of NE in $\mathbb{X}_{\rm RG}(\zeta)$; i.e., in~\cite{Facchinei.Kanzow:07}.
\end{proof}


\subsection{Approximation errors of the RG game}

As mentioned before, in the RG game, the hard constraints $\mu_{\ell}\geq0$, $\lambda^{[i]}_{\ell}\geq0$ and $z^{[i]}\in Z_i$ are relaxed by those defined via the logarithmic function. Hence, the RG game is completely unconstrained. This will allow us to characterize the sensitivity of the RG game on $\zeta$. However, the RG game is merely an approximation of the CVX game. We now move to characterize how good this approximation is.

By the convexity or concavity of $\LL_i$ on its components, the following first-order conditions hold for the unique NE $\tilde{\eta}(\zeta)\in\mathbb{X}_{\rm RG}(\zeta)$:
\begin{align}&\nabla_{z^{[i]}}\LL_i(\tilde{z},\tilde{\mu},\tilde{\lambda}^{[i]},\zeta) = 0,\quad \nabla_{\lambda^{[i]}}\LL_i(\tilde{z},\tilde{\mu},\tilde{\lambda}^{[i]},\zeta) = 0,\nnum\\
&\nabla_{\mu}\HH(\tilde{z},\tilde{\mu},\zeta) = 0.\label{e76}\end{align}

These relations are explicitly expressed as follows:
\begin{align}&\nabla_{\beta^{[i]}_{\kappa}} f_i(\tilde{z}^{[i]}) + \sum_{\ell=1}^m\tilde{\mu}_{\ell} \nabla_{\beta^{[i]}_{\kappa}} G_{\ell}(\tilde{\beta},\zeta) + \sum_{\ell=1}^{p}\tilde{\lambda}^{[i]}_{\ell} \nabla_{\beta^{[i]}_{\kappa}} h^{[i]}_{\ell}(\tilde{z}^{[i]})\nnum\\
&- \frac{\tau}{\tilde{\beta}^{[i]}_{\kappa}} + \frac{\tau}{\beta^{[i]}_{\max} - \tilde{\beta}^{[i]}_{\kappa}}  = 0,\quad \kappa\in \mathbb{S},\label{e77}\\
&\nabla_{\alpha^{[i]}_{\kappa\kappa'}} f_i(\tilde{z}^{[i]}) + \sum_{\ell=1}^{p}\tilde{\lambda}^{[i]}_{\ell} \nabla_{\alpha^{[i]}_{\kappa\kappa'}} h^{[i]}_{\ell}(\tilde{z}^{[i]})\nnum\\&- \frac{\tau}{\tilde{\alpha}^{[i]}_{\kappa\kappa'}} + \frac{\tau}{\alpha^{[i]}_{\max} - \tilde{\alpha}^{[i]}_{\kappa\kappa'}}  = 0,\quad (\kappa,\kappa')\in \mathcal{E}_{\mathbb{S}},\label{e78}\\
&\epsilon\tilde{\mu}_{\ell} - G_{\ell}(\tilde{\beta},\zeta) - \frac{\tau}{\tilde{\mu}_{\ell}} = 0,\quad \ell=1,\cdots,m,\label{e83}\\
&\epsilon\tilde{\lambda}^{[i]}_{\ell} - h^{[i]}_{\ell}(\tilde{z}^{[i]}) - \frac{\tau}{\tilde{\lambda}^{[i]}_{\ell}} = 0,\quad \ell=1,\cdots,p.\label{e79}\end{align}

Since $\tilde{\mu}_{\ell},\tilde{\lambda}^{[i]}_{\ell}>0$, solving~\eqref{e83} and~\eqref{e79} renders the following: \begin{align}\tilde{\mu}^{[i]}_{\ell} = g(G^{[i]}_{\ell}(\tilde{\beta},\zeta)) > 0, \quad \tilde{\lambda}^{[i]}_{\ell} = g(h^{[i]}_{\ell}(\tilde{z}^{[i]})) > 0.\label{e71}\end{align}

The following proposition verifies that the RG game can be rendered arbitrarily close to the CVX game by choosing a pair of sufficiently small $\epsilon$ and $\tau$. In particular, (P1) and (P2) show that the violation of equality constraints is at most $\max\{\varsigma_h(\epsilon,\tau),\varsigma_G(\epsilon,\tau)\} = o(\epsilon,\tau)$. (P3) implies that the cost at the NE $\tilde{\eta}(\zeta)$ is $o(\tau)$-suboptimal. (P4) and (P5) provide a set of bounds on NEs.

\begin{proposition} The unique NE $\tilde{\eta}(\zeta)\in \mathbb{X}_{\rm RG}(\zeta)$ is an approximation of $\mathbb{X}_{\rm C}(\zeta)$ in the following way:
\begin{enumerate}
\item[(P1)] $|h^{[i]}_{\kappa}(\tilde{z}^{[i]})| \leq \varsigma_h(\epsilon,\tau)$;
\item[(P2)] $|G_{\kappa}(\tilde{\beta})| \leq \varsigma_G(\epsilon,\tau)$;
\item[(P3)] The following holds for any $z^{[i]}\in Z_i$ with $G(\beta^{[i]},\tilde{\beta}^{[-i]},\zeta)\leq 0$ and $h^{[i]}(z^{[i]})\leq 0$: \begin{align*}f_i(\tilde{z}^{[i]})
&\leq f_i(z^{[i]})+\tau(p+m)\\
&+2\tau(|\mathbb{S}|\log(\beta^{[i]}_{\max})
+|\mathcal{E}_{\mathbb{S}}|\log(\alpha^{[i]}_{\max})).\end{align*}
\item[(P4)] It holds that for $i\in V$: \begin{align}&g(-\varsigma_G(\epsilon,\tau))\leq \tilde{\mu}_{\kappa}(\zeta) \leq g(\varsigma_G(\epsilon,\tau)),\nnum\\
&g(-\varsigma_h(\epsilon,\tau))\leq \tilde{\lambda}^{[i]}_{\kappa}(\zeta) \leq g(\varsigma_h(\epsilon,\tau)).\nnum\end{align}
\item[(P5)] It holds that
    \begin{align*}&\frac{\tau\beta^{[i]}_{\max}}{2\tau+\delta_i'\beta^{[i]}_{\max}}
    \leq\tilde{\beta}^{[i]}_{\kappa}\leq \beta^{[i]}_{\max}-\frac{\tau\beta^{[i]}_{\max}}{2\tau+\delta_i'\beta^{[i]}_{\max}},\\
    &\frac{\tau\alpha^{[i]}_{\max}}{2\tau+\delta_i''\alpha^{[i]}_{\max}}
    \leq\tilde{\alpha}^{[i]}_{\kappa\kappa'}\leq \alpha^{[i]}_{\max}-\frac{\tau\alpha^{[i]}_{\max}}{2\tau+\delta_i''\alpha^{[i]}_{\max}}.\end{align*}
\end{enumerate}
\label{pro1}
\end{proposition}

\begin{proof} In order to simplify the notations, we drop the dependency on the parameter $\zeta$ unless necessary.

\textbf{Claim 1:} (P1) holds.

\begin{proof} Since $\tilde{\eta}$ is an NE, then the following holds for any $z^{[i]}\in\real^{n_i}$: \begin{align}\LL_i(\tilde{z},\tilde{\mu},\tilde{\lambda}_i)
-\LL_i(z^{[i]},\tilde{z}^{[-i]},\tilde{\mu},\tilde{\lambda}_i)\leq0,\nnum\end{align}
that is, \begin{align}&f_i(\tilde{z}^{[i]}) + \langle\tilde{\mu}, G(\tilde{\beta})\rangle + \langle\tilde{\lambda}^{[i]},h^{[i]}(\tilde{z}^{[i]})\rangle\nnum\\
&-\tau\sum_{\kappa\in \mathbb{S}}\big(\psi(\tilde{\beta}^{[i]}_{\kappa})+\psi(\beta^{[i]}_{\max}
-\tilde{\beta}^{[i]}_{\kappa})\big)\nnum\\ &-\tau\sum_{(\kappa,\kappa')\in \mathcal{E}_{\mathbb{S}}}\big(\psi(\tilde{\alpha}^{[i]}_{\kappa\kappa'})
+\psi(\alpha^{[i]}_{\max}-\tilde{\alpha}^{[i]}_{\kappa\kappa'})\big)\nnum\\
&\leq f_i(z^{[i]}) + \langle\tilde{\mu}, G(\beta^{[i]},\tilde{\beta}^{[-i]})\rangle +\langle\tilde{\lambda}^{[i]},h^{[i]}(z^{[i]})\rangle\nnum\\& -\tau\sum_{\kappa\in \mathbb{S}}\big(\psi(\beta^{[i]}_{\kappa})+\psi(\beta^{[i]}_{\max}-\beta^{[i]}_{\kappa})\big)\nnum\\ &-\tau\sum_{(\kappa,\kappa')\in \mathcal{E}_{\mathbb{S}}}\big(\psi(\alpha^{[i]}_{\kappa\kappa'})
+\psi(\alpha^{[i]}_{\max}-\alpha^{[i]}_{\kappa\kappa'})\big).\label{e44}\end{align}

Choose $\hat{\beta}^{[i]} = \tilde{\beta}^{[i]}$. By Assumption~\ref{asm4}, there is $\hat{\alpha}$ with $\hat{\alpha}^{[i]}_{\kappa\kappa'}\in[a,\alpha^{[i]}_{\max}-a]$ such that $h^{[i]}(\hat{z}^{[i]})\leq0$. Substitute $\hat{z}^{[i]}$ into~\eqref{e44}, and we have
\begin{align}&f_i(\tilde{z}^{[i]}) + \langle\tilde{\lambda}^{[i]},h^{[i]}(\tilde{z}^{[i]})\rangle\nnum\\
&-\tau\sum_{(\kappa,\kappa')\in \mathcal{E}_{\mathbb{S}}}\big(\psi(\tilde{\alpha}^{[i]}_{\kappa\kappa'})
+\psi(\alpha^{[i]}_{\max}-\tilde{\alpha}^{[i]}_{\kappa\kappa'})\big)\nnum\\
&\leq f_i(\hat{z}^{[i]}) +\langle\tilde{\lambda}^{[i]},h^{[i]}(\hat{z}^{[i]})\rangle\nnum\\ &-\tau\sum_{(\kappa,\kappa')\in \mathcal{E}_{\mathbb{S}}}\big(\psi(\hat{\alpha}^{[i]}_{\kappa\kappa'})
+\psi(\alpha^{[i]}_{\max}-\hat{\alpha}^{[i]}_{\kappa\kappa'})\big).\label{e45}\end{align}

Notice that the last two terms on the right-hand side of~\eqref{e45} are non-positive. So it follows from~\eqref{e45} that
\begin{align}\langle\tilde{\lambda}^{[i]},h^{[i]}(\tilde{z}^{[i]})\rangle\leq \delta_i.\label{e73}\end{align}

By~\eqref{e79}, it is easy to see that $\tilde{\lambda}^{[i]}_{\kappa}h^{[i]}_{\kappa}(\tilde{z}^{[i]})$ are lower bounded by $-\tau$. Substitute these relations into~\eqref{e73}, and it gives that \begin{align}\tilde{\lambda}^{[i]}_{\kappa}h^{[i]}_{\kappa}(\tilde{z}^{[i]}) \leq \delta_i + (p-1)\tau,\label{e75}\end{align}

Consider the first case of $h^{[i]}_{\kappa}(\tilde{z}^{[i]}) \geq 0$. Substitute~\eqref{e71} into~\eqref{e75}, and it renders that \begin{align}2h^{[i]}_{\kappa}(\tilde{z}^{[i]})^2&\leq h^{[i]}_{\kappa}(\tilde{z}^{[i]})^2
+h^{[i]}_{\kappa}(\tilde{z}^{[i]})\sqrt{h^{[i]}_{\kappa}(\tilde{z}^{[i]})^2+4\epsilon\tau}\nnum\\
&\leq 2\epsilon(\delta_i + (p-1)\tau).\nnum\end{align} Hence, $h^{[i]}_{\kappa}(\tilde{z}^{[i]})\leq \varsigma_h(\epsilon,\tau)$.

Consider the second case of $h^{[i]}_{\kappa}(\tilde{z}^{[i]}) \leq 0$. By the asymmetry of the components in $h^{[i]}$, there is $\kappa'\neq \kappa$ such that $h^{[i]}_{\kappa}(\tilde{z}^{[i]}) = -h^{[i]}_{\kappa'}(\tilde{z}^{[i]})$ and $h^{[i]}_{\kappa'}(\tilde{z}^{[i]})\geq 0$. Follow the above steps, and we have $h^{[i]}_{\kappa'}(\tilde{z}^{[i]})\leq \varsigma_h(\epsilon,\tau)$. Hence, we have $h^{[i]}_{\kappa}(\tilde{z}^{[i]})\geq -\varsigma_h(\epsilon,\tau)$. The combination of the above two cases establishes (P1).
\end{proof}

\textbf{Claim 2:} (P2) holds.

\begin{proof} Pick any $1\leq\kappa\leq |\mathbb{S}|$ and let $\kappa' = \kappa+\frac{m}{2}$. Then $G_{\kappa}(\tilde{\beta}) = -G_{\kappa'}(\tilde{\beta})$.

Consider the first case of $G_{\kappa}(\tilde{\beta})\geq0$. Recall that $\sum_{i\in V}\tilde{\beta}^{[i]}_{\kappa} - \zeta_{\kappa} = G_{\kappa}(\tilde{\beta})$. So there exists $i\in V$ such that $\frac{1}{N}(\zeta_{\kappa}+G_{\kappa}(\tilde{\beta})) \leq \tilde{\beta}^{[i]}_{\kappa}$. Choose $\hat{\beta}^{[i]}$ such that $\hat{\beta}^{[i]}_{\kappa} = \frac{\zeta_{\kappa}}{N}$ and $\hat{\beta}^{[i]}_{\kappa'} = \tilde{\beta}^{[i]}_{\kappa'}$ for $\kappa'\neq\kappa$. Recall that $\frac{\zeta_{\kappa}}{N}\in[a,\frac{\beta_{\max}}{N}-a]$. Hence, $\hat{\beta}^{[i]}_{\kappa}\in[a,\beta^{[i]}_{\max}-a]$. By Assumption~\ref{asm4}, there is $\hat{\alpha}^{[i]}$ such that $\hat{\alpha}^{[i]}_{\kappa}\in[a,\alpha_{\max}^{[i]}-a]$ and $h^{[i]}(\hat{z}^{[i]})\leq0$.

By~\eqref{e71} and $G_{\kappa}(\tilde{\beta}) = -G_{\kappa'}(\tilde{\beta})$, we have $\tilde{\mu}_{\kappa}-\tilde{\mu}_{\kappa'} = \frac{1}{\epsilon}G_{\kappa}(\tilde{\beta})$.
Hence, we have \begin{align}&\langle\tilde{\mu},G(\tilde{\beta})-G(\hat{\beta})\rangle\nnum\\
&= \tilde{\mu}_{\kappa}(G_{\kappa}(\tilde{\beta}_{\kappa})-G_{\kappa}(\hat{\beta}))
+\tilde{\mu}_{\kappa'}(G_{\kappa'}(\tilde{\beta}_{\kappa'})-G_{\kappa'}(\hat{\beta}))\nnum\\
&= \tilde{\mu}_{\kappa}\frac{G_{\kappa}(\tilde{\beta})}{N} + \tilde{\mu}_{\kappa'}(-\frac{G_{\kappa}(\tilde{\beta})}{N}) = \frac{G_{\kappa}(\tilde{\beta})^2}{N\epsilon}.\label{e72}\end{align}

Recall that $\tilde{\lambda}^{[i]}_{\kappa}h^{[i]}_{\kappa}(\tilde{z}^{[i]})\geq-\tau$. Substitute $\hat{z}^{[i]}$ into~\eqref{e44}, and it gives that $G_{\kappa}(\tilde{\beta})^2\leq N\epsilon(\delta_i + p\tau)$. Hence, $G_{\kappa}(\tilde{\beta})\leq \sqrt{N\epsilon(\delta_i + p\tau)}$ when $G_{\kappa}(\tilde{\beta})\geq0$.

Analogous to Claim 1, we have $G_{\kappa}(\tilde{\beta})\geq -\sqrt{N\epsilon(\delta_i + p\tau)}$ when $G_{\kappa}(\tilde{\beta})\leq0$. The combination of the above two cases establishes (P2).
\end{proof}

\textbf{Claim 3:} (P3) holds.

\begin{proof} It is a result of the relation~\eqref{e44}.
\end{proof}

\textbf{Claim 4:} (P4) holds.

\begin{proof} It is easy to verify that the first-order derivative $g'(s) = \frac{g(s)}{\sqrt{s+\sqrt{s^2+4\epsilon\tau}}} > 0$ and the second-order derivative $g''(s) = \frac{2s^2+4\epsilon\tau}{(s^2+4\epsilon\tau)^{\frac{3}{2}}} > 0$. Hence, the function $g$ is strictly increasing and strictly convex. From~\eqref{e71} and (P1), we establish the desired bounds on $\tilde{\mu}_{\kappa}$ and $\tilde{\lambda}^{[i]}_{\kappa}$.
\end{proof}

\textbf{Claim 5:} (P5) holds.

\begin{proof} It is noted that \begin{align}&\|\nabla_{\beta^{[i]}_{\kappa}} f_i(\tilde{z}^{[i]}) + \sum_{\ell=1}^m\tilde{\mu}_{\ell} \nabla_{\beta^{[i]}_{\kappa}} G_{\ell}(\tilde{\beta},\zeta) + \sum_{\ell=1}^{p}\tilde{\lambda}^{[i]}_{\ell} \nabla_{\beta^{[i]}_{\kappa}} h^{[i]}_{\ell}(\tilde{z}^{[i]})\|\nnum\\
&\leq \sup_{z^{[i]}\in Z_i}\|\nabla_{\beta^{[i]}_{\kappa}} f_i(z^{[i]})\| + \sum_{\ell=1}^m\|\tilde{\mu}_{\ell}\| + \sum_{\ell=1}^{p}\|\tilde{\lambda}^{[i]}_{\ell}\|\leq \delta_i'.\end{align}

Assume $\tilde{\beta}^{[i]}_{\kappa} < \frac{\tau\beta^{[i]}_{\max}}{2\tau+\delta_i'\beta^{[i]}_{\max}} < \frac{\beta^{[i]}_{\max}}{2}$. Then we have \begin{align}\delta_i'-\frac{\tau}{\tilde{\beta}^{[i]}_{\kappa}} + \frac{\tau}{\beta^{[i]}_{\max}-\tilde{\beta}^{[i]}_{\kappa}} < \delta_i' - \frac{\tau}{\tilde{\beta}^{[i]}_{\kappa}} +\frac{2\tau}{\beta^{[i]}_{\max}}\leq0.\nnum\end{align} This contradicts~\eqref{e77}. So it must be $\tilde{\beta}^{[i]}_{\kappa} \geq \frac{\tau\beta^{[i]}_{\max}}{2\tau+\delta_i'\beta^{[i]}_{\max}}$. The remainder of (P5) can be shown in an analogous way.
\end{proof}

It completes the proof for Proposition~\ref{pro1}.
\end{proof}

\begin{remark} The bounds on $\beta_{\kappa}^{[i]}$ and $\alpha^{[i]}_{\kappa\kappa'}$ are shifted by $a$, and the argument in $\psi$ is scaled by $a$. In this way the last two terms on the right-hand side of~\eqref{e45} are non-positive for any $z^{[i]}\in Z_i$.\oprocend\label{rem6}
\end{remark}

\subsection{Sensitivity analysis}

As mentioned before, the RG game is completely unconstrained. This allows us to perform sensitivity analysis on the RG game, and characterize the variation of the NE $\tilde{\eta}(U(\xi))$ induced by the variation of $\xi$. In this part, we will drop the dependency of the NE $\tilde{\eta}(\zeta)$ on $\zeta$ unless necessary.

Toward this end, we denote a set of matrices as follows: \begin{align*}&R_1(\tilde{\eta}) \triangleq \left[\begin{array}{ccc}
\nabla_{z^{[1]}z^{[1]}}^2\LL_1 & \cdots & \nabla_{z^{[1]}z^{[N]}}^2\LL_1 \\
\vdots & \ddots & \vdots \\
\nabla_{z^{[1]}z^{[1]}}^2\LL_N & \cdots & \nabla_{z^{[N]}z^{[1]}}^2\LL_N \\
\end{array}\right],\nnum\\
&R_2 \triangleq \left[\begin{array}{ccc}
\nabla_{z^{[1]}}G_1(\tilde{\beta},\zeta)^T & \cdots & \nabla_{z^{[1]}}G_m(\tilde{\beta},\zeta)^T \\
\vdots & \ddots & \vdots \\
\nabla_{z^{[N]}}G_1(\tilde{\beta},\zeta)^T & \cdots & \nabla_{z^{[N]}}G_m(\tilde{\beta},\zeta)^T \\
\end{array}\right],\nnum\\
&R_3 \triangleq \diag{[\nabla_{z^{[i]}} h^{[i]}_1(\tilde{z}^{[i]})^T,\cdots,\nabla_{z^{[i]}} h^{[i]}_{p}(\tilde{z}^{[i]})^T]}_{i\in V},\nnum\\
&R_4(\tilde{\eta}) \triangleq\diag{\epsilon+\frac{\tau}{\tilde{\mu}_{\ell}^2}}_{\ell=1,\cdots,m},\nnum\\
&R_5(\tilde{\eta}) \triangleq \diag{\diag{\epsilon
+\frac{\tau}{(\tilde{\lambda}^{[i]}_{\ell})^2}}_{\ell=1,\cdots,p}}_{i\in V}.\end{align*}

Recall that $G$ and $h^{[i]}$ are affine. Then $\nabla_{z^{[j]}}\LL_i = 0$ if $i\neq j$. Since $\LL_i$ is separable in its components, thus $R_1(\tilde{\eta},\zeta)$ is diagonal, symmetric and positive definite. In addition, $R_2$ and $R_3$ are constant due to $G$ and $h^{[i]}$ being affine.

With the above notations at hand, we can derive the partial derivative of the left-hand side of~\eqref{e76} with respect to $\tilde{\eta}$ evaluated at $(\tilde{\eta},\zeta)$, and this derivative is given by:
\begin{align}J_M(\tilde{\eta}) \triangleq \left[\begin{array}{ccc}
                                                R_1(\tilde{\eta}) & R_2^T & R_3^T\\
                                                -R_2 & R_4(\tilde{\eta}) & 0\\
                                                -R_3 & 0 & R_5(\tilde{\eta})
                                              \end{array}\right].\nnum
\end{align}

Let $J_N$ be the partial derivative of the left-hand side of~\eqref{e77} to~\eqref{e79} with respect to $\zeta$. Since $G$ is affine in $\zeta$, then $J_N$ is state-independent. We then denote \begin{align}J(\tilde{\eta})\triangleq
J_M(\tilde{\eta})^{-1}J_N,\label{e48}\end{align} where $J_M(\tilde{\eta})^{-1}$ will be shown to be non-singular in the following lemma.

\begin{lemma} The matrix $J_M(\tilde{\eta}(\zeta))$ is non-singular, positive definite and its spectrum is uniformly lower bounded by $\epsilon\min_{i\in V}\{\rho_i,\rho_i'\} > 0$. In addition, $J(\tilde{\eta}(\zeta))$ is continuously differential in $\zeta$, and the following relation holds: \begin{align}\frac{d \tilde{\eta}(U(\xi(t)))}{d t}=J(\tilde{\eta}(U(\xi(t))))\frac{d U(\xi(t))}{d \xi(t)}\dot{\xi}(t).\label{e4}\end{align}\label{lem4}
\end{lemma}

\begin{proof} Recall the following identity for non-singular $A_1$:
\begin{align}\left[
\begin{array}{cc}
A_1 & A_2 \\
A_3 & A_4 \\
\end{array}
\right] = \left[
\begin{array}{cc}
A_1 & 0 \\
A_3 & I \\
\end{array}
\right]\left[
\begin{array}{cc}
I & A_1^{-1}A_2 \\
0 & A_4 - A_3A_1^{-1}A_2 \\
\end{array}
\right].\label{e46}
\end{align}

By~\eqref{e46}, the determinant of $J_M(\tilde{\eta})$ is computed as follows: \begin{align}{\rm det}(J_M(\tilde{\eta})) = {\rm det}\left[\begin{array}{cc}R_4(\tilde{\eta}) & 0\\0 & R_5(\tilde{\eta})\end{array}\right]{\rm det}(T_1(\tilde{\eta})),\nnum\end{align} where $T(\tilde{\eta})$ is given by: \begin{align*}T_1(\tilde{\eta})&\triangleq R_1(\tilde{\eta}) + T_2(\tilde{\eta})\\ &\triangleq R_1(\tilde{\eta})+[R_2^T\;\;R_3^T]\left[\begin{array}{cc}R_4(\tilde{\eta}) & 0\\0 & R_5(\tilde{\eta})\end{array}\right]^{-1}\left[
                                                 \begin{array}{c}
                                                   R_2 \\
                                                   R_3 \\
                                                 \end{array}
                                               \right].\end{align*}

Recall that $R_1(\tilde{\eta})$, $R_4(\tilde{\eta})$ and $R_5(\tilde{\eta})$ are symmetric and diagonal. So $T_1(\tilde{\eta})$ and $T_2(\tilde{\eta})$ are symmetric. Since $R_4(\tilde{\eta})$ and $R_5(\tilde{\eta})$ are positive definite and diagonal, so $T_2(\tilde{\eta})$ is positive semi-definite. Recall $R_1(\tilde{\eta})$ is positive definite. By using $\lambda_{\min}(A_1+A_2)\geq \lambda_{\min}(A_1) + \lambda_{\min}(A_2)$, we know that $T_1(\tilde{\eta})$ is positive definite. Hence, ${\rm det}(T_1(\tilde{\eta}))\neq0$ and ${\rm det}(J_M(\tilde{\eta}))\neq0$. This implies that $J_M(\tilde{\eta})$ is non-singular.

By~\eqref{e46} again, the determinant of $J_M(\tilde{\eta})$ is computed as follows: \begin{align}\lambda_{\min}(J_M(\tilde{\eta})) &\geq \lambda_{\min}\big(\left[\begin{array}{cc}R_4(\tilde{\eta}) & 0\\0 & R_5(\tilde{\eta})\end{array}\right]\big)\lambda_{\min}(T_1(\tilde{\eta}))\nnum\\
&\geq \epsilon\lambda_{\min}(R_1(\tilde{\eta})) \geq \epsilon\min_{i\in V}\{\rho_i,\rho_i'\}.\nnum\end{align}

Recall that $J_M(\tilde{\eta})$ is non-singular. By the inverse function theorem, we reach that $\tilde{\eta}(\zeta)$ and $J(\tilde{\eta}(\zeta))$ are continuously differentiable in $\zeta$ and the derivative of $\tilde{\eta}$ with respect to~$\zeta$ is given by: \begin{align}\frac{d \tilde{\eta}(\zeta)}{d \zeta} = J(\tilde{\eta}(\zeta)).\label{e82}\end{align}
With the relation~\eqref{e82}, we establish the derivative of $\tilde{\eta}(U(\xi(t)))$ with respect to $t$ as follows: \begin{align}\frac{d \tilde{\eta}(U(\xi(t)))}{d t} = \frac{d \tilde{\eta}(\zeta(t))}{d \zeta(t)}\dot{\zeta}(t)=J(\tilde{\eta}(\zeta(t)))\frac{d U(\xi(t))}{d \xi(t)}\dot{\xi}(t),\nnum\end{align} where $\frac{d U(\xi(t))}{d \xi(t)}$ is well-defined since $U$ is smooth.
\end{proof}

\begin{remark} In the paper~\cite{Chen.Lau:12}, a relation like~\eqref{e4} between saddle-points and the parameter is derived from the Karush-Kuhn-Tucker condition. However, the results in~\cite{Chen.Lau:12} are not applicable to our problem. Firstly, the Lagrangian functions $\LL_i$ and $\HH$ are merely concave in $\lambda^{[i]}$ and $\mu$ if $\epsilon, \tau = 0$. Secondly, the paper~\cite{Chen.Lau:12} assumes that the state-dependent matrix derived from the Karush-Kuhn-Tucker condition is uniformly non-singular. This is not easy to check \emph{a priori} and may lead to instability in our feedback setup.\oprocend\label{rem4}
\end{remark}

\begin{lemma} The functions $J(\eta)\frac{d U(\xi)}{d \xi}$ and $\nabla\Omega$ are Lipschitz continuous with constant $L_J > 0$ and $L_{\Omega} > 0$, respectively, over $Y$.\label{lem5}
\end{lemma}

\begin{proof} Note that $J_M(\eta)J_M(\eta)^{-1}=I$. Take the derivative on $\eta$, and we have \begin{align}\frac{d J_M(\eta)}{d \eta}J_M(\eta)^{-1} + J_M(\eta)\frac{d J_M(\eta)^{-1}}{d \eta} = 0.\nnum\end{align} This gives the following relations: \begin{align}\|\frac{d J_M(\eta)^{-1}}{d \eta}\| &\leq \|J_M(\eta)^{-1}\|\|\frac{d J_M(\eta)}{d \eta}\|\|J_M(\eta)^{-1}\|\nnum\\
&\leq (\epsilon\min_{i\in V}\{\rho_i,\rho_i'\})^{-2}\|\frac{d J_M(\eta)}{d \eta}\|,\nnum\end{align} where in the last inequality we use Lemma~\ref{lem4}.

By~\eqref{e74}, we derive the following relations: \begin{align}&\|\frac{d J(\eta)}{d \eta}\frac{d U(\xi)}{d \xi}\|\leq \|\frac{d J_M(\eta)^{-1}}{d \eta}\|\|J_N\|\|\frac{d U(\xi)}{d \xi}\|\nnum\\
&\leq (\epsilon\min_{i\in V}\{\rho_i,\rho_i'\})^{-2}\|\frac{d J_M(\eta)}{d \eta}\|\|\|J_N\|\|\frac{d U(\xi)}{d \xi}\|,\nnum\\
&\|J(\eta)\frac{d^2 U(\xi)}{d^2 \xi}\|\leq\|J(\eta)\|\|\frac{d^2 U(\xi)}{d^2 \xi}\|\nnum\\
&\leq \epsilon\min_{i\in V}\{\rho_i,\rho_i'\}\|J_N\|D_U^{(2)}.\label{e55}\end{align} From~\eqref{e55}, we reach the desired Lipschitz constant $L_J$ on $J(\eta)\frac{d U(\xi)}{d \xi}$ over $Y$.
\end{proof}

\section{Real-time game theoretic coordination}\label{sec:main}

In this section, we will present an algorithm for the real-time game theoretic coordination. It will be followed by the convergence properties of the closed-loop system.

\subsection{Algorithm statement}

Denote by $\tilde{\eta}(t)$ the NE of the CVX game parameterized by $\zeta(t) = U(\xi(t))$ where $\tilde{\eta}(t)$ consists of \begin{align*}&\tilde{\beta}(t) \triangleq ((\tilde{\beta}^{[i]}_{\kappa}(t))_{\kappa\in\mathbb{S}})_{i\in V},\quad\tilde{\alpha}(t) \triangleq ((\tilde{\alpha}^{[i]}_{\kappa\kappa'}(t))_{(\kappa,\kappa')\in \mathcal{E}_{\mathbb{S}}})_{i\in V},\\
&\tilde{\mu}(t)\triangleq (\tilde{\mu}_{\kappa}(t))_{\kappa\in\{1,\cdots,m\}},\quad \tilde{\lambda}(t) \triangleq ((\tilde{\lambda}^{[i]}_{\kappa}(t))_{\kappa\in\{1,\cdots,p\}})_{i\in V}.\end{align*}

At each time~$t$, each primal player~$i$ maintains the estimates $(\hat{\beta}^{[i]}_{\kappa}(t))_{\kappa\in\mathbb{S}}$ and $(\hat{\alpha}^{[i]}_{\kappa\kappa'}(t))_{(\kappa,\kappa')\in \mathcal{E}_{\mathbb{S}}}$ of $(\tilde{\beta}^{[i]}_{\kappa}(t))_{\kappa\in\mathbb{S}}$ and $(\tilde{\alpha}^{[i]}_{\kappa\kappa'}(t))_{(\kappa,\kappa')\in \mathcal{E}_{\mathbb{S}}}$. Each dual player~$i$ maintains the estimates $(\tilde{\lambda}^{[i]}_{\kappa}(t))_{\kappa\in\{1,\cdots,p\}}$ of $(\lambda^{[i]}_{\kappa}(t))_{\kappa\in\{1,\cdots,p\}}$. The operator then maintains the estimate $\mu(t)$ of $\tilde{\mu}(t)$.

The decision makers update their own estimates by decreasing the distance to the instantaneous Nash equilibrium $\tilde{\eta}(t)$ and simultaneously following the temporal variation of $\tilde{\eta}(t)$. The update rules are given in Algorithm~\ref{ta:algo2}. In particular, the quantity $v(t) \triangleq J(\eta(t),U(\xi(t)))\frac{d U(\xi(t))}{d \xi(t)}\dot{\xi}(t)$ serves as the estimate of $\frac{d\tilde{\eta}(U(\xi(t)))}{d t}$ in Lemma~\ref{lem4}, and is decomposed into $(v^{[1]}_p(t)^T,\cdots,v^{[N]}_p(t)^T, v_{\mu}(t)^T, v^{[1]}_d(t)^T,\cdots,v^{[N]}_d(t)^T)^T$ where $v^{[i]}_p(t)$ (resp. $v^{[i]}_d(t)$) is assigned to primal (resp. dual) player~$i$ and $v_{\mu}(t)$ is assigned to the operator.

The quantities of $\hat{\beta}^{[i]}(t)$ and $\hat{\alpha}^{[i]}(t)$ are intermediate estimates and probably fails to enforce the constraint $\dot{v}_{\kappa}(t) = 0$. To address this, each primal player~$i$ obtains $(\alpha^{[i]}(t),\beta^{[i]}(t)) = \mathbb{Q}_i(b^{[i]}(t))$ where $\beta^{[i]}(t) = \hat{\beta}^{[i]}(t)$ and $\alpha^{[i]}(t)$ is the orthogonal projection of $\hat{\alpha}^{[i]}(t)$ onto the set $\Xi^{[i]}(t)$ defined by: \begin{align*}\Xi^{[i]}(t) \triangleq &\{\alpha^{[i]}\in\real^{|\mathcal{E}_{\mathbb{S}}|}\;|\;A \alpha^{[i]} = b^{[i]}(t),\\ &\;\;\alpha^{[i]}_{\kappa\kappa'}\in[a,\alpha^{[i]}_{\max}-a],\quad (\kappa,\kappa')\in\EE_{\mathbb{S}}\}\end{align*} where $b^{[i]}_{\kappa}(t) = \beta^{[i]}_{\kappa}(t) - \sum_{\kappa'\in\NN_{\kappa}}a_{\kappa'\kappa}\beta^{[i]}_{\kappa'}(t)$ and $b^{[i]}(t) = (b^{[i]}_{\kappa}(t))_{\kappa\in\mathbb{S}}$. After that, all the players in $V$ implements the control commands $\beta^{[i]}_{\kappa}(t)$ and $\alpha^{[i]}_{\kappa\kappa'}(t)$ in the queue dynamics~\eqref{e18}. Here, each primal player~$i$ prioritizes the stabilization of the user queues, and enforces the constraint $\dot{v}_{\kappa}(t) = 0$ by reallocating his empty vehicles.

\begin{remark} The computation of the orthogonal projection can be encoded by the following quadratic program which can be solved by a number of existing efficient algorithms: \begin{align*}&\min_{\alpha^{[i]}\in\real^{|\mathcal{E}_{\mathbb{S}}|}}\|\alpha^{[i]} - \hat{\alpha}^{[i]}(t)\|^2\\
&{\rm s.t.}\quad A \alpha^{[i]} = b^{[i]}(t), \quad \alpha^{[i]}_{\kappa\kappa'}\in[a,\alpha^{[i]}_{\max}-a],\quad (\kappa,\kappa')\in\EE_{\mathbb{S}}.\end{align*}

If $\alpha_{\max}$ is sufficiently large, the computation of the orthogonal projection can be greatly simplified. Let $\alpha^{[i]}(t)$ to the projection of $\hat{\alpha}^{[i]}(t)$ onto the solution set of $A \alpha^{[i]} = b^{[i]}(t)$. Notice that $A$ is orthogonal to all the vectors in the plane of $A \alpha^{[i]} = b^{[i]}(t)$. Then we have \begin{align*}A(A\hat{\alpha}^{[i]}(t)- A\alpha^{[i]}(t)) = \hat{\alpha}^{[i]}(t)- \alpha^{[i]}(t).\end{align*} That is, \begin{align*}\alpha^{[i]}(t) = \hat{\alpha}^{[i]}(t) -A(A\hat{\alpha}^{[i]}(t)- b^{[i]}(t)).\end{align*}\label{rem3}\oprocend
\end{remark}

%
%

The real-time game theoretic regulator is formally stated in Algorithm~\ref{ta:algo2}.

\begin{algorithm}[htbp]\caption{Real-time game theoretic coordinator} \label{ta:algo2}
\begin{algorithmic}[1]
\REQUIRE Each primal (resp. dual) player~$i$ chooses the initial state $z^{[i]}(0)\in Z_i$ (resp. $\lambda^{[i]}(0)\in \Lambda_i$). The operator chooses the initial state $\mu(0)\in M$.
\ENSURE At each time instant $t \ge 0$, the decision makers execute the following steps:
\end{algorithmic}

\begin{algorithmic}[1]
\STATE The operator generates the estimate $v(t) = J(\eta(t),U(\xi(t)))\frac{d U(\xi(t))}{d \xi(t)}\dot{\xi}(t)$, and update $\mu(t)$ according to the following rule:\begin{align*}
\dot{\mu}(t) &= \mathbb{P}_M[\mu(t) + \alpha D_{\mu}(t) + v_{\mu}(t)] - \mu(t),
\end{align*} where $\alpha > 0$ and $D_{\mu}(t) \triangleq \nabla_{\mu}\HH(z(t),\mu(t),U(\xi(t)))$. The operator then informs player~$i$ the information of $(v^{[i]}_p(t),v^{[i]}_d(t),\mu(t))$.
\STATE Each primal (resp. dual) player $i$ updates $z^{[i]}(t) = (\hat{\beta}^{[i]}(t),\hat{\alpha}^{[i]}(t))$ (resp. $\lambda^{[i]}(t)$) according to the following rule:\begin{align*} \dot{z}^{[i]}(t) &= \mathbb{P}_{Z_i}[z^{[i]}(t) - \alpha D^{[i]}_z(t) + v^{[i]}_p(t)] - z^{[i]}(t),\\
\dot{\lambda}^{[i]}(t) &= \mathbb{P}_{\Lambda_i}[\lambda^{[i]}(t) + \alpha D^{[i]}_{\lambda}(t) + v^{[i]}_d(t)] - \lambda^{[i]}(t),
\end{align*} where $D^{[i]}_x(t) \triangleq \nabla_{x^{[i]}}\LL_i(z(t),\mu(t),\lambda^{[i]}(t),U(\xi(t)))$ and $D^{[i]}_{\mu}(t) \triangleq \nabla_{\mu^{[i]}}\LL_i(z(t),\mu(t),\lambda^{[i]}(t),U(\xi(t)))$.
\STATE Each player~$i$ generates and implements the control commands $(\alpha^{[i]}(t),\beta^{[i]}(t)) = \mathbb{Q}_i(b^{[i]}(t))$.
\end{algorithmic}
\end{algorithm}

\subsection{The closed-loop system and its performance analysis}

The closed-loop system consists of the user queueing network~\eqref{e18},~\eqref{e14}, the vehicle queueing network~\eqref{e6} and the real-time game theoretic coordinator (Algorithm~\ref{ta:algo2}). For the sake of completeness, we summarize the closed-loop system in Algorithm~\ref{ta:algo3}.

\begin{algorithm}[htbp]\caption{The closed-loop system} \label{ta:algo3}

\begin{algorithmic}[1]
\STATE The dynamics of the user queueing network: \begin{align*}\dot{Q}(t) &= h_Q(c(t),u(t)),\\
\dot{c}(t) &= h_c(c(t),Q(t),t),\end{align*} where the controller $u_{\kappa}(t) = \sum_{i\in V}\beta^{[i]}_{\kappa}(t)$ with $(\alpha^{[i]}(t),\beta^{[i]}(t)) = \mathbb{Q}_i(b^{[i]}(t))$.
\STATE The dynamics of the vehicle queueing network: \begin{align*}\dot{v}_{\kappa}(t) = 0.\end{align*}
\STATE The real-time game theoretic coordinator: \begin{align*}\dot{\eta}(t) &= \mathbb{P}_{K}[\eta(t)-\alpha D(t) + J(\eta(t),\zeta(t))\frac{d U(\xi(t))}{d \xi(t)}\dot{\xi}(t)]\nnum\\
&- \eta(t).
\end{align*}
\end{algorithmic}
\end{algorithm}

%


The following theorem summarizes the performance of the closed-loop system.

\begin{theorem} Suppose Assumptions~\ref{asm1},~\ref{asm7} and~\ref{asm4} hold. Suppose the following holds: \begin{align}\vartheta &\triangleq \big(1-\alpha\rho_{\Omega}+\alpha^2L_{\Omega}^2+(L_JD_u^{(1)}\delta_{\xi})^2\nnum\\
&+(1+\alpha L_{\Omega})L_JD_u^{(1)}\delta_{\xi}\big)^{\frac{1}{2}}<1.\nnum\end{align} The estimates $\beta^{[i]}(t)$ and $\alpha^{[i]}(t)$ generated by Algorithm~\ref{ta:algo2} approximates $\tilde{\eta}(t) = {\mathbb{X}}_{\rm RG}(\zeta(t))$ in the way that \begin{align} &\lim_{t\rightarrow+\infty}\|\beta^{[i]}(t) - \tilde{\beta}^{[i]}(t)\| \rightarrow 0,\nnum\\
&\limsup_{t\rightarrow+\infty}\|\alpha^{[i]}(t) - \tilde{\alpha}^{[i]}(t)\| \leq \|A\|\varsigma_G(\epsilon,\tau),\label{e41}\end{align} Furthermore, the queue dynamics achieves the following: \begin{align}\limsup_{t\rightarrow+\infty}\|Q_{\kappa}(t) - \bar{Q}_{\kappa}\|\leq \max\{\Delta_{\min},\Delta_{\max}\},\label{e40}\end{align} where $\Delta_{\min} \triangleq \ln (1+\frac{2\varsigma_G(\epsilon,\tau)}{\beta_{\max}-c_{\max}-a}
(1+\frac{\beta_{\max}-c_{\max}-a}{c_{\min}}e^{\bar{Q}_{\kappa}}))$ and $\Delta_{\max} \triangleq -\ln (1-\frac{2\varsigma_G(\epsilon,\tau)}{\beta_{\max}-c_{\max}-a})$.\label{the2}
\end{theorem}

\begin{remark} Recall that $\rho_{\Omega} = \min\{\min_{i\in V}\{\rho_i,\rho_i'\},\epsilon\}$ and $\delta_{\xi}$ represents an upper bound on the maximum variation of the user arrival rate. If $\delta_{\xi}$ is smaller, we can choose a set of smaller $\alpha$, $\epsilon$ and $\tau$ to satisfy $\vartheta < 1$ and reduce the right-hand sides of~\eqref{e41} and~\eqref{e40}. That is, if the maximum variation of the user arrival rates is smaller, the steady-state system performance can be improved.\oprocend\label{rem6}
\end{remark}

\begin{proof} We divide the proof into several claims.

\textbf{Claim 1:} It holds that $\|v(t)\| \leq \delta_v$ for all $t\geq0$.

\begin{proof} It is noted that \begin{align}&\|v(t)\|\leq \|J(\eta(t),U(\xi(t)))\|\|\frac{d U(\xi(t))}{d \xi(t)}\|\|\dot{\xi}(t)\|\nnum\\
&\leq \frac{\|J_N\| D_u^{(1)}\delta_{\xi}}{\epsilon\min_{i\in V}\{\rho_i,\rho_i'\}} = \delta_v,\nnum\end{align} where we use Lemma~\ref{lem2}, the relation~\eqref{e48} and Lemma~\ref{lem4}.
\end{proof}

\textbf{Claim 2:} It holds that $\mu_{\kappa}(t)\in [0,\Delta_{\mu}]$ and $\lambda_{\kappa}^{[i]}(t)\in [0,\Delta_{\lambda}]$ for all $t\geq0$.

\begin{proof} The dynamics associated with $\mu_{\kappa}$ can be written as follows: \begin{align}\dot{\mu}_{\kappa}(t) = -\mu_{\kappa}(t) + d_{\kappa}(t),\label{e30}
\end{align} where $d_{\kappa}(t)\in[0,\Delta_{\mu}]$ for all $t\geq0$. It is readily to see that $[0,\Delta_{\mu}]$ is an invariant set of~\eqref{e30} and thus $\mu_{\kappa}(t)\in[0,\Delta_{\mu}]$ for all $t\geq0$. Analogously, one can verify that $\alpha_{\kappa}^{[i]}(t)\in[0,\Delta_{\alpha}]$ for all $t\geq0$.\end{proof}

\textbf{Claim 3:} It holds that $\delta^{[i]}_{\beta,\min}
\leq\beta^{[i]}_{\kappa}(t)\leq \delta^{[i]}_{\beta,\max}$ and
$\delta^{[i]}_{\alpha,\min}
\leq\alpha^{[i]}_{\kappa\kappa'}(t)\leq \delta^{[i]}_{\alpha,\max}$ for all $i\in V$ and $t\geq0$.

\begin{proof} The dynamics associated with $\beta_{\kappa}^{[i]}$ can be written as follows: \begin{align}&\dot{\beta}_{\kappa}^{[i]}(t) = \frac{\tau}{\beta^{[i]}_{\kappa}(t)} - \frac{\tau}{\beta_{\max} - \beta^{[i]}_{\kappa}(t)} + d^{[i]}_{\kappa}(t),\label{e56}\end{align}
where
\begin{align}d^{[i]}_{\kappa}(t) &= -\nabla_{\beta^{[i]}_{\kappa}} f_i(z^{[i]}(t)) + \sum_{\ell=1}^m\mu_{\ell}(t) \nabla_{\beta^{[i]}_{\kappa}} G_{\ell}(\beta(t),U(\xi(t)))\nnum\\
&+ \sum_{\ell=1}^{p}\lambda^{[i]}_{\ell}(t) \nabla_{\beta^{[i]}_{\kappa}} h^{[i]}_{\ell}(z^{[i]}(t)) + v^{[i]}_{\kappa}(t).\nnum\end{align}

Note that $d^{[i]}_{\kappa}(t) \leq d_{\beta}$. Analogous to (P5) of Proposition~\ref{pro1}, one can show that $\delta^{[i]}_{\beta,\min}
\leq\beta^{[i]}_{\kappa}(t)\leq \delta^{[i]}_{\beta,\max}$. Analogously, it holds that
$\delta^{[i]}_{\alpha,\min}
\leq\alpha^{[i]}_{\kappa\kappa'}(t)\leq \delta^{[i]}_{\alpha,\max}$.
\end{proof}

\textbf{Claim 4:} It holds that $\mu_{\kappa}(t)\in[\delta_{\mu,\min},\delta_{\mu,\max}]$ and $\lambda_{\kappa}(t)\in[\delta_{\lambda,\min},\delta_{\lambda,\max}]$.

\begin{proof} The dynamics associated with $\mu_{\kappa}$ can be written as follows: \begin{align}\dot{\mu}_{\kappa}(t) &= \mathbb{P}_M[\mu_{\kappa}(t)-(\epsilon\mu_{\kappa}(t) - G_{\kappa}(\beta(t),\zeta_{\kappa}(t)) - \frac{\tau}{\mu_{\kappa}(t)})\nnum\\
&+ v_{\kappa}(t)] -\mu_{\kappa}(t).\end{align}\end{proof}

Let $u(t) = \nabla\Omega(\tilde{\eta}(t),\zeta(t))$, and $\chi(t) \triangleq \frac{d U(\xi(t))}{d \xi(t)}$. From the definition of $K$ one can see that $\mathbb{P}_{K}$ never applies and thus, \begin{align}&\mathbb{P}_{K}[\tilde{\eta}(t)-\alpha u(t) + J(\tilde{\eta}(t))\chi(t)\dot{\xi}(t)]\nnum\\
&= \tilde{\eta}(t)-\alpha u(t) + J(\tilde{\eta}(t))\chi(t)\dot{\xi}(t).\nnum\end{align}
This implies the following temporal evolution of $\tilde{\eta}(t)$: \begin{align}&\frac{d\tilde{\eta}(t)}{d t} = J(\tilde{\eta}(t))\chi(t)\dot{\xi}(t)\nnum\\
&= \mathbb{P}_{K}[\tilde{\eta}(t)-\alpha u(t) + J(\tilde{\eta}(t))\chi(t)\dot{\xi}(t)] - \tilde{\eta}(t).\label{e1}\end{align}

The combination of Lemma~\ref{lem5} and Claims 5 and 6 implies that
\begin{align}&\|J(\eta(t))-J(\tilde{\eta}(t))\| \leq L_J \|\eta(t)-\tilde{\eta}(t)\|,\label{e64}\\
&\|D(t)-u(t)\| \leq L_{\Omega}\|\eta(t)-\tilde{\eta}(t)\|.\label{e53}\end{align}

We recall the regulator for players as follows: \begin{align}\dot{\eta}(t) = \mathbb{P}_{K}[\eta(t)-\alpha D(t) + J(\eta(t),\zeta(t))\chi(t)\dot{\xi}(t)] - \eta(t).\label{e39}
\end{align}

Choose the Lyapunov function candidate $W(\eta(t),\tilde{\eta}(t)) \triangleq \frac{1}{2}\|\eta(t)-\tilde{\eta}(t)\|^2$ for system~\eqref{e39}. The following claim provides an estimate of $\dot{W}$.

\textbf{Claim 5:} The following estimate holds: \begin{align}\dot{W} &\leq 2(-1+(\vartheta+\sigma)^{\frac{1}{2}})W+\Upsilon(t),\label{e24}\end{align}

\begin{proof} It follows from~\eqref{e1} and~\eqref{e39} that \begin{align}\dot{W} &= \langle \eta(t)-\tilde{\eta}(t), \dot{\eta}(t)-\frac{d\tilde{\eta}(t)}{d t}\rangle\nnum\\
&=-\|\eta(t)-\tilde{\eta}(t)\|^2 + \Psi(t),\label{e36}\end{align}
where the term $\Psi(t)$ is given by: \begin{align*}&\Psi(t)\triangleq \langle \eta(t)-\tilde{\eta}(t),\nnum\\ &\mathbb{P}_{K}[\eta(t)-\alpha D(t) + J(\eta(t),\zeta(t))\chi(t)\dot{\xi}(t)]\nnum\\
&- \mathbb{P}_{K}[\tilde{\eta}(t)-\alpha u(t) + J(\tilde{\eta}(t))\chi(t)\dot{\xi}(t)]\rangle.\end{align*}

By the non-expansiveness property of $\mathbb{P}_{K}$, we have \begin{align}&\|\mathbb{P}_{K}[\eta(t)-\alpha D(t) + J(\eta(t),\zeta(t))\chi(t)\dot{\xi}(t)]\nnum\\
&- \mathbb{P}_{K}[\tilde{\eta}(t)-\alpha u(t) + J(\tilde{\eta}(t))\chi(t)\dot{\xi}(t)]\|^2\nnum\\
&\leq \|(\eta(t)-\tilde{\eta}(t))-\alpha (D(t)-u(t))\nnum\\ &+(J(\eta(t),\zeta(t))\chi(t)\dot{\xi}(t)
-J(\tilde{\eta}(t))\chi(t)\dot{\xi}(t))\|^2\nnum\\
&\leq \|\eta(t)-\tilde{\eta}(t)\|^2 -\alpha\langle \eta(t)-\tilde{\eta}(t), D(t)-u(t)\rangle\nnum\\
&+ \alpha^2\|D(t)-u(t)\|^2\nnum\\
&+\|J(\eta(t),\zeta(t))\chi(t)\dot{\xi}(t)-J(\tilde{\eta}(t))\chi(t)\dot{\xi}(t)\|^2\nnum\\
&+\|\eta(t)-\tilde{\eta}(t)\|\nnum\\
&\times\|J(\eta(t),\zeta(t))\chi(t)\dot{\xi}(t)
-J(\tilde{\eta}(t))\chi(t)\dot{\xi}(t)\|\nnum\\
&+\alpha\|D(t)-u(t)\|\nnum\\
&\times\|J(\eta(t),\zeta(t))\chi(t)\dot{\xi}(t)
-J(\tilde{\eta}(t))\chi(t)\dot{\xi}(t)\|.\label{e2}\end{align}

By the strong monotonicity, we have the following for the second term on the right-hand side of~\eqref{e2}: \begin{align}\langle \eta(t)-\tilde{\eta}(t), -(D(t) - u(t))\rangle \leq -\rho_{\Omega}\|\eta(t)-\tilde{\eta}(t)\|^2.\label{e52}\end{align}

We have the following for the last three terms on the right-hand side of~\eqref{e2}:
\begin{align}&\|J(\eta(t),\zeta(t))\chi(t)\dot{\xi}(t)
-J(\tilde{\eta}(t))\chi(t)\dot{\xi}(t)\|\nnum\\
&\leq L_J\|\chi(t)\dot{\xi}(t)\|\|\eta(t)-\tilde{\eta}(t)\|.\label{e3}\end{align}

Substitute~\eqref{e52},~\eqref{e53} and~\eqref{e3} into~\eqref{e2}. After grouping, we have the following: \begin{align}&\|\mathbb{P}_{K}[\eta(t)-\alpha D(t) + J(\eta(t),\zeta(t))\chi(t)\dot{\xi}(t)]\nnum\\
&- \mathbb{P}_{K}[\tilde{\eta}(t)-\alpha u(t) + J(\tilde{\eta}(t))\chi(t)\dot{\xi}(t)]\|^2\nnum\\
&\leq \big(1-\alpha\rho_{\Omega}+\alpha^2L_{\Omega}^2
+L_J^2(D_u^{(1)}\delta_{\xi})^2\nnum\\
&+(1+\alpha L_{\Omega})L_JD_u^{(1)}\delta_{\xi}\big)\|\eta(t)-\tilde{\eta}(t)\|^2.\end{align}
This estimate further implies the following estimate of $\Psi(t)$: \begin{align}&\|\Psi(t)\|\leq(1-\alpha\rho_{\Omega}
+\alpha^2L_{\Omega}^2+L_J^2(D_u^{(1)}\delta_{\xi})^2\nnum\\
&+(1+\alpha L_{\Omega})L_J(D_u^{(1)}\delta_{\xi}))^{\frac{1}{2}}
\|\eta(t)-\tilde{\eta}(t)\|^2.\label{e8}\end{align}
\end{proof}

%

Claim 4 immediately implies that the following convergence property: \begin{align}\lim_{t\rightarrow+\infty}\|\eta(t)-\tilde{\eta}(t)\| = 0.\label{e25}\end{align} From Proposition~\ref{pro1}, it follows that $\tilde{\eta}(t)\in \mathbb{X}_{\rm RG}(\zeta(t))$ is an approximation of $\mathbb{X}_{\rm C}(\zeta(t))$. Since the function $[G_{\kappa}(\cdot)]^+$ is continuous, thus it follows from Proposition~\ref{pro1} and~\eqref{e25} that \begin{align}\limsup_{t\rightarrow+\infty}\|G_{\kappa}(x(t),\zeta(t))\| \leq \varsigma_G(\epsilon,\tau).\label{e26}\end{align} Hence, the controller for the queue network can be written in the following way: \begin{align}u_{\kappa}(t) = U_{\kappa}(\xi(t)) + \Delta_{\kappa}(t),\label{e27}\end{align} where the perturbation term $\Delta(t)$ satisfies \begin{align*}\limsup_{t\rightarrow+\infty}\|\Delta_{\kappa}(t)\| \leq \varsigma_G(\epsilon,\tau).\end{align*}

\textbf{Claim 7:} The relation~\eqref{e40} holds.

\begin{proof} Consider the Lyapunov function candidate $V_{\kappa}(Q_{\kappa}) = \frac{1}{2}(Q_{\kappa}-\bar{Q}_{\kappa})^2$. Its Lie derivative along~\eqref{e10} is given by: \begin{align}\frac{d V_{\kappa}}{d Q_{\kappa}}\dot{Q}_{\kappa} &= (Q_{\kappa}(t)-\bar{Q}_{\kappa})\dot{Q}_{\kappa}\nnum\\
&= (Q_{\kappa}(t)-\bar{Q}_{\kappa})(\hat{U}_{\kappa}(\xi(t))+\Delta_{\kappa}(t))\nnum\\
&= (Q_{\kappa}(t)-\bar{Q}_{\kappa})
\big(\chi(e^{-(Q_{\kappa}(t)-\bar{Q}_{\kappa})}-1)+\Delta_{\kappa}(t)\big).\label{e32}
\end{align}

When $Q_{\kappa}(t)-\bar{Q}_{\kappa} \geq \Delta_{\max}$, we have \begin{align}&\frac{\beta_{\max}-c_{\max}-a}{1+\frac{\beta_{\max}-c_{\max}-a}{c_{\min}}e^{-(Q_{\kappa}(t)-\bar{Q}_{\kappa})}}
(e^{-(Q_{\kappa}(t)-\bar{Q}_{\kappa})}-1)+\Delta(t)\nnum\\
&\leq (\beta_{\max}-c_{\max}-a)(e^{-(Q_{\kappa}(t)-\bar{Q}_{\kappa})}-1) + \varsigma_G(\epsilon,\tau)\leq -\varsigma_G(\epsilon,\tau).\nnum\end{align} Hence, the following holds for $Q_{\kappa}(t)-\bar{Q}_{\kappa} \geq -\ln (1-\frac{2\varsigma_G(\epsilon,\tau)}{\beta_{\max}-c_{\max}-a})$: \begin{align}\frac{d V_{\kappa}}{d Q_{\kappa}}\dot{Q}_{\kappa} \leq -\varsigma_G(\epsilon,\tau)\Delta_{\max}.\label{e33}\end{align}

Analogously, when $Q_{\kappa}(t)-\bar{Q}_{\kappa} \leq -\Delta_{\min}$, we have \begin{align}&\frac{\beta_{\max}-c_{\max}-a}{1+\frac{\beta_{\max}-c_{\max}-a}{c_{\min}}e^{-(Q_{\kappa}(t)-\bar{Q}_{\kappa})}}
(e^{-(Q_{\kappa}(t)-\bar{Q}_{\kappa})}-1)+\Delta_{\kappa}(t)\nnum\\
&\geq \frac{\beta_{\max}-c_{\max}-a}{1+\frac{\beta_{\max}-c_{\max}-a}{c_{\min}}e^{-\bar{Q}_{\kappa}}}
(e^{-(Q_{\kappa}(t)-\bar{Q}_{\kappa})}-1) - \varsigma_G(\epsilon,\tau)\nnum\\
&\geq \varsigma_G(\epsilon,\tau).\nnum\end{align} Hence, the following holds for $Q_{\kappa}(t)-\bar{Q}_{\kappa} \leq -\Delta_{\min}$: \begin{align}\frac{d V_{\kappa}}{d Q_{\kappa}}\dot{Q}_{\kappa} \leq -\varsigma_G(\epsilon,\tau)\Delta_{\min}.\label{e34}\end{align}

The combination of~\eqref{e33},~\eqref{e34} establishes the desired result of~\eqref{e40}.
\end{proof}
This completes the proof for Theorem~\ref{the2}.
\end{proof}

\section{Conclusions}

In this paper, we have introduced a model of competitive MoD systems and proposed a real-time game theoretic coordination problem for the system. We have came up with an algorithm to achieve vehicle balance and practical regulation of the user queueing network.

\section{Appendix}\label{sec:appendix}

\subsection{Notations}

In this section, we summarize the notations used in Sections~\ref{sec:preliminaries} and~\ref{sec:main}.

\subsubsection{Notations for Section~\ref{sec:preliminaries}}

Denote $Z_i \triangleq \{z^{[i]}\in\real^{n_i}\;|\;\beta^{[i]}_{\kappa}\in[a,\beta^{[i]}_{\max}-a],\quad \kappa\in\mathbb{S},\quad\alpha^{[i]}_{\kappa\kappa'}\in[a,\alpha^{[i]}_{\max}-a],\quad (\kappa,\kappa')\in\EE_{\mathbb{S}}\}$, $Z\triangleq \prod_{i\in V}Z_i$, $\hat{Z}_i \triangleq \{z^{[i]}\in\real^{n_i}\;|\;\beta^{[i]}_{\kappa}\in[0,\beta^{[i]}_{\max}],\quad \kappa\in\mathbb{S},\quad\alpha^{[i]}_{\kappa\kappa'}\in[0,\alpha^{[i]}_{\max}],\quad (\kappa,\kappa')\in\EE_{\mathbb{S}}\}$ and $\hat{Z}\triangleq \prod_{i\in V}\hat{Z}_i$.

\begin{align*} &D_U^{(1)} \triangleq |\mathbb{S}|(1+\frac{\beta_{\max}-c_{\max}+\frac{c_{\min}}{2}-Na}{(1+2\frac{\beta_{\max}-c_{\max}-Na}{c_{\min}})^2}\nnum\\
&\times \frac{2(\beta_{\max}-c_{\max}-Na)}{c_{\min}}e^{\max_{\kappa\in \mathbb{S}}\bar{Q}_{\kappa}}),\\
&D_U^{(2)} \triangleq \frac{(\beta_{\max}-c_{\max}+\frac{c_{\min}}{2}-Na)\frac{2(\beta_{\max}-c_{\max}-Na)}{c_{\min}}}{(1+2\frac{\beta_{\max}-c_{\max}-Na}{c_{\min}})^4}\nnum\\
&\times(1+(\frac{2(\beta_{\max}-c_{\max}-Na)}{c_{\min}}e^{\max_{\kappa\in\mathbb{S}}\bar{Q}_{\kappa}})^2)e^{\max_{\kappa\in\mathbb{S}}\bar{Q}_{\kappa}},\\
\end{align*}

\begin{align*}g(s)&\triangleq \frac{1}{2\epsilon}(s+\sqrt{s^2+4\epsilon\tau}),\\
\delta_i&\triangleq \sup_{z^{[i]}\in Z_i}f_i(z^{[i]}) - \inf_{z^{[i]}\in Z_i}f_i(z^{[i]})\\&+ 2\tau(|\mathbb{S}|\psi(\beta_{\max})+|\mathcal{E}_{\mathbb{S}}|\psi(\alpha_{\max})),\\ \varsigma_h(\epsilon,\tau)&\triangleq\max_{i\in V}\sqrt{\epsilon(\delta_i+(p-1)\tau)},\\
\varsigma_G(\epsilon,\tau)&\triangleq\max_{i\in V}\sqrt{N\epsilon(\delta_i+p\tau)},\\
\delta_i'&\triangleq \sup_{z^{[i]}\in Z_i}\|\nabla_{\beta^{[i]}_{\kappa}} f_i(z^{[i]})\|\\
&+ 2|\mathbb{S}|\max\{g(-\varsigma_G(\epsilon,\tau)),g(\varsigma_G(\epsilon,\tau))\}\\
&+ 2|\mathbb{S}|\max\{g(-\varsigma_h(\epsilon,\tau)),g(\varsigma_h(\epsilon,\tau))\},\\
\delta_i''&\triangleq \sup_{z^{[i]}\in Z_i}\|\nabla_{\alpha^{[i]}_{\kappa\kappa'}} f_i(z^{[i]})\|\\
&+2|\mathbb{S}|\max\{g(-\varsigma_h(\epsilon,\tau)),g(\varsigma_h(\epsilon,\tau))\}.\end{align*}

\begin{align*}&\delta_v \triangleq (\epsilon\min_{i\in V}\{\rho_i,\rho_i'\})^{-1}\|J_N\| D_u^{(1)}\delta_{\xi},\\
&\Delta_{\mu} \triangleq g(\varsigma_G(\epsilon,\tau)) + (\epsilon\min_{i\in V}\{\rho_i,\rho_i'\})^{-1}\|J_N\|D_u^{(1)}\delta_{\xi},\\ &\Delta_{\lambda} \triangleq g(\varsigma_h(\epsilon,\tau)) + (\epsilon\min_{i\in V}\{\rho_i,\rho_i'\})^{-1}\|J_N\|D_u^{(1)}\delta_{\xi},\\
d_{\beta} &\triangleq \max_{i\in V}\sup_{z^{[i]}\in Z_i}\|\nabla_{\beta^{[i]}_{\kappa}} f_i(z^{[i]})\| + 2|\mathbb{S}|(\Delta_{\mu} + \Delta_{\lambda}),\\
d_{\alpha} &\triangleq \max_{i\in V}\sup_{z^{[i]}\in Z_i}\|\nabla_{\alpha^{[i]}_{\kappa\kappa'}} f_i(z^{[i]})\| + 2|\mathbb{S}|\Delta_{\lambda},\\
d_{\mu}&\triangleq \max_{\kappa\in\mathbb{S}}\sup_{z\in Z,\zeta_\kappa\in[\frac{c_{\min}}{2},\beta_{\max}-Na]}
G_{\kappa}(\beta,\zeta_{\kappa})+\delta_v,\\
d_{\lambda}&\triangleq \max_{i\in V}\max_{\kappa\in\mathbb{S}}\sup_{z^{[i]}\in Z_i}h^{[i]}_{\kappa}(z)+\delta_v.\end{align*}

\begin{align*}
&\delta^{[i]}_{\beta,\min} \triangleq \frac{\tau\beta^{[i]}_{\max}}{2\tau+d_{\beta}\beta^{[i]}_{\max}},\quad
\delta^{[i]}_{\beta,\max} = \beta^{[i]}_{\max}-\frac{\tau\beta^{[i]}_{\max}}{2\tau+d_{\beta}\beta^{[i]}_{\max}},\\
&\delta^{[i]}_{\alpha,\min} \triangleq \frac{\tau\alpha^{[i]}_{\max}}{2\tau+d_{\alpha}\alpha^{[i]}_{\max}},\quad
\delta^{[i]}_{\alpha,\max} = \alpha^{[i]}_{\max}-\frac{\tau\alpha^{[i]}_{\max}}{2\tau+d_{\alpha}\alpha^{[i]}_{\max}},\\
&\delta_{\mu,\min} \triangleq \frac{\tau}{\epsilon+d_{\mu}},\quad
\delta_{\mu,\max} = \Delta_{\mu},\\
&\delta_{\lambda,\min} \triangleq \frac{\tau}{\epsilon+d_{\lambda}},\quad
\delta_{\lambda,\max} = \Delta_{\lambda},\\
Y&\triangleq\{\eta \in \real^n \; | \; \beta^{[i]}_{\kappa}\in [\delta^{[i]}_{\beta,\min},\delta^{[i]}_{\beta,\max}],\quad\alpha^{[i]}_{\kappa\kappa'}\in [\delta^{[i]}_{\alpha,\min},\delta^{[i]}_{\alpha,\max}],\\
&\mu_{\ell}\in[\delta_{\mu,\min},\delta_{\mu,\max}],\quad
\lambda^{[i]}_{\kappa}\in[\delta_{\lambda,\min},\delta_{\lambda,\max}]\}.\end{align*}

\begin{align*} L_J &\triangleq \sqrt{\big(\sup_{\eta\in Y}\|\frac{d R_1(\eta)}{d \eta}\|\big)^2 + 4\tau^2(\frac{m^2}{\varsigma_G^6(\epsilon,\tau)}+\frac{N^2p^2}{\varsigma_h^6(\epsilon,\tau)})}\\
&\times (\epsilon\min_{i\in V}\{\rho_i,\rho_i'\})^{-2}\|J_N\|D_U^{(1)} + \epsilon\min_{i\in V}\{\rho_i,\rho_i'\}\|J_N\|D_U^{(2)},\\
L_{\Omega} &\triangleq \sup_{\eta\in Y}\|\frac{d\nabla \Omega(\eta)}{d\eta}\|.\end{align*}

\subsubsection{Notations for Section~\ref{sec:main}}

We associate the incidence matrix $A \in \real^{|\mathbb{S}|\times2|\mathbb{S}|}$ for the graph $\GG_{\mathbb{S}}$. In particular, the $\kappa$-th row is assigned to state $\kappa$, and is in the form of $[a_{\kappa 1} \cdots a_{\kappa |\mathbb{S}|},\; a_{1 \kappa} \cdots a_{|\mathbb{S}| \kappa}]$. If $\kappa'\in \NN_{\kappa}$, then $a_{\kappa\kappa'} = -1$; if $\kappa'\in \NN_{\kappa}$, then $a_{\kappa'\kappa} = 1$; $a_{\kappa\kappa'} = 0$, otherwise.

\begin{align*}&\Lambda_i \triangleq \{\lambda^{[i]}\in\real^p \; | \; \lambda^{[i]}_{\kappa}\in[0,\Delta_{\lambda}],\quad \forall \kappa \in \mathbb{S}\},\\
&M \triangleq \{\mu\in\real^m \; | \; \mu_{\kappa}\in[0,\Delta_{\mu}],\quad \forall \kappa \in \mathbb{S}\}.\end{align*}

\subsection{An instrumental result}

The following lemma shows that the infinite-horizon averages of two functions are identical if two functions asymptotically approach to each other.

\begin{lemma} Consider the functions $f, g : \real_{\geq0} \rightarrow \real$ which are uniformly bounded. If $\lim_{t\rightarrow+\infty}\|f(t)-g(t)\| = 0$, then it holds that $\lim_{T\rightarrow+\infty}\|\bar{f}(T)-\bar{g}(T)\|=0$, where
$\bar{f}(T) = \frac{\int_0^Tf(t)dt}{T}$ and $\bar{g}(T) = \frac{\int_0^Tg(t)dt}{T}$.
\label{lem3}\end{lemma}
\begin{proof} Pick any $\varepsilon > 0$, there is $K\geq0$ such that $\|f(t)-g(t)\|\leq\varepsilon$ for all $t\geq K$. Then the following holds for any $T\geq K$ \begin{align}\|\bar{f}(T)-\bar{g}(T)\| &\leq \frac{\int_0^K\|f(t)-g(t)\|dt}{T}+\frac{\int_K^T\|f(t)-g(t)\|dt}{T}\nnum\\
&\leq \frac{\int_0^K\|f(t)-g(t)\|dt}{T}+\frac{\varepsilon(T-K)}{T}.\label{e42}\end{align}

Recall that $f,g$ are uniformly bounded. Take the limit on $T$ in~\eqref{e42}, and it renders that \begin{align}\limsup_{T\rightarrow+\infty}\|\bar{f}(T)-\bar{g}(T)\| \leq \varepsilon.\label{e43}\end{align} Since~\eqref{e43} holds for any $\varepsilon>0$, we then reach the desired result.
\end{proof}

\bibliographystyle{plain}
\bibliography{ZhuFrazzoliACC13}

\begin{thebibliography}{10}

\bibitem{Alpcan.Basar:05}
T.~Alpcan and T.~Ba\c{s}ar.
\newblock A utility-based congestion control scheme for {I}nternet-style
  networks with delay.
\newblock {\em IEEE Transactions on Networking}, 13(6):1261 -- 1274, 2005.

\bibitem{Altman.Basar:98}
E.~Altman and T.Ba\c{s}ar.
\newblock Multi-user rate-based flow control.
\newblock {\em IEEE Transactions on Communications}, 46(7):940--949, 1998.

\bibitem{Arrow.Debreu:54}
K.J. Arrow and G.~Debreu.
\newblock Existence of an equilibrium for a competitive economy.
\newblock {\em Econometrica}, 22:265--290, 1954.

\bibitem{Basar.Olsder:82}
T.~Ba\c{s}ar and G.J. Olsder.
\newblock {\em Dynamic Noncooperative Game Theory}.
\newblock Academic Press, 1982.

\bibitem{Cavalcante.Rogers.Jennings.Yamada:11}
R.L.G. Cavalcante, A.~Rogers, N.R. Jennings, and I.~Yamada.
\newblock Distributed asymptotic minimization of sequences of convex functions
  by a broadcast adaptive subgradient method.
\newblock {\em IEEE Journal of Selected Topics in Signal Processing}, 5(4):739
  -- 753, 2011.

\bibitem{Chen.Lau:12}
J.~Chen and K.N. Lau.
\newblock Convergence analysis of saddle point problems in time varying
  wireless systems - control theoretical approach.
\newblock {\em IEEE Transactions on Signal Processing}, 60(1):443 -- 452, 2012.

\bibitem{Chen.Lau.Chen:11}
J.~Chen, K.N. Lau, and Y.~Chen.
\newblock Distributive network utility maximization ({NUM}) over time-varying
  fading channels.
\newblock {\em IEEE Transactions on Signal Processing}, 59(5):2395 -- 2404,
  2011.

\bibitem{Chen.Li.Jiang.Low:12}
L.~Chen, N.~Li, L.~Jiang, and S.H. Low.
\newblock {\em Control and Optimization Methods for Electric Smart Grids},
  volume~3, chapter~3, pages 63--86.
\newblock Springer, 2012.

\bibitem{Facchinei.Kanzow:07}
F.~Facchinei and C.~Kanzow.
\newblock Generalized {N}ash equilibrium problems.
\newblock {\em Journal of Operations Research}, 5(3):173--210, 2007.

\bibitem{Facchinei.Pang:03}
F.~Facchinei and J.-S. Pang.
\newblock {\em Finite-dimensional variational inequalities and complementarity
  problems}.
\newblock Springer-Verlag, 2003.

\bibitem{Gordon.Greenwald.Marks:08}
G.J. Gordon, A.~Greenwald, and C.~Marks.
\newblock No-regret learning in convex games.
\newblock In {\em Proceedings of the 25th international conference on Machine
  learning}, pages 360 -- 367, 2008.

\bibitem{Ma.Callaway.Hiskens:12}
Z.~Ma, D.S. Callaway, and I.A. Hiskens.
\newblock Decentralized charging control of large populations of plug-in
  electric vehicles.
\newblock {\em IEEE Transactions on Control Systems Technology}, 2012.
\newblock To appear.

\bibitem{Marden.Wierman:08}
J.~Marden and A.~Wierman.
\newblock Distributed welfare games with applications to sensor coverage.
\newblock In {\em Proc. IEEE Conf. on Decision and Control}, pages 1708--1713,
  Cancun, Mexico, December 2008.

\bibitem{Mayne.Rawlins.Rao.Scokaert:00}
D.Q. Mayne, J.B. Rawlings, C.V. Rao, and P.O.~M. Scokaert.
\newblock Constrained model predictive control: stability and optimality.
\newblock {\em Automatica}, 36:789--814, 2000.

\bibitem{Palomar.Eldar:10}
D.P. Palomar and Y.C. Eldar.
\newblock {\em Convex optimization in signal processing and communications}.
\newblock Cambridge University Press, 2010.

\bibitem{Pan.Pavel:09}
Y.~Pan and L.~Pavel.
\newblock Games with coupled propagated constraints in optical network with
  multi-link topologies.
\newblock {\em Automatica}, 45(7):871 -- 880, 2009.

\bibitem{Pang.Scutari.Facchinei.Wang:08}
J.-S. Pang, G.~Scutari, F.~Facchinei, and C.~Wang.
\newblock Distributed power allocation with rate constraints in {G}aussian
  parallel interference channels.
\newblock {\em IEEE Transactions on Information Theory}, 54(8):3471--3489,
  2008.

\bibitem{Pavone.Smith.Frazzoli.Rus:12}
M.~Pavone, S.L. Smith, E.~Frazzoli, and D.~Rus.
\newblock Robotic load balancing for mobility-on-demand systems.
\newblock {\em International Journal of Robotics Research}, 31(7):839 -- 854,
  2012.

\bibitem{Rosen:65}
J.B. Rosen.
\newblock Existence and uniqueness of equilibrium points for concave {N}-person
  games.
\newblock {\em Econometrica}, 33(3):520--534, 1965.

\bibitem{Schrank.Lomax.Eisele:11}
D.~Schrank, T.~Lomax, and B.~Eisele.
\newblock 2011 urban mobility report.
\newblock Technical report, Texas Transportation Institute, 2011.

\bibitem{Yin.Shanbhag.Mehta:11}
H.~Yin, U.V. Shanbhag, and P.G. Mehta.
\newblock Nash equilibrium problems with scaled congestion costs and shared
  constraints.
\newblock {\em IEEE Transactions on Automatic Control}, 56(7):1702--1708, 2011.

\bibitem{Zhu.Frazzoli:12}
M.~Zhu and E.~Frazzoli.
\newblock On distributed equilibrium seeking for generalized convex games.
\newblock In {\em IEEE Conference on Decision and Control}, pages 4858--4863,
  Maui, HI, 2012.

\bibitem{Zhu.Martinez:12}
M.~Zhu and S.~Mart{\'\i}nez.
\newblock Distributed coverage games for energy-aware mobile sensor networks.
\newblock {\em SIAM Journal on Control and Optimization}, 51(1):1--27, 2013.

\bibitem{Zinkevich:03}
M.~Zinkevich.
\newblock Online convex programming and generalized infinitesimal gradient
  ascent.
\newblock In {\em Proceedings of the Twentieth International Conference on
  Machine Learning}, 2003.

\end{thebibliography}

\end{document}